\documentclass[10pt]{amsart}
\usepackage{amsmath}
\usepackage{amscd,amsthm,amssymb,amsfonts}
\usepackage{mathrsfs}
\usepackage{dsfont}
\usepackage{stmaryrd}
\usepackage{euscript}
\usepackage{expdlist}
\usepackage{enumerate}

\input xy
\xyoption{all}
\usepackage[OT2,T1]{fontenc}

\theoremstyle{plain}
\newtheorem{thm}{Theorem}[section]
\newtheorem{thmA}{Theorem}

\newtheorem*{thm*}{Theorem}

\newtheorem{lm}[thm]{Lemma}
\newtheorem{cor}[thm]{Corollary}
\newtheorem*{cor*}{Corollary}
\newtheorem{prop}[thm]{Proposition}
\newtheorem*{conj*}{Conjecture}

\theoremstyle{remark}
\newtheorem*{remark}{Remark}
\newtheorem*{thank}{Acknowledgments}

\theoremstyle{definition}
\newtheorem*{defn*}{Definition}
\newtheorem{Remark}[thm]{Remark}
\newtheorem{defn}[thm]{Definition}

\newcommand{\nc}{\newcommand}

\newcommand{\beq}{\begin{equation}}
\newcommand{\eeq}{\end{equation}}
\newcommand{\bpmx}{\begin{pmatrix}}
\newcommand{\epmx}{\end{pmatrix}}
\newcommand{\bbmx}{\begin{bmatrix}}
\newcommand{\ebmx}{\end{bmatrix}}
\newcommand{\wh}{\widehat}
\newcommand{\wtd}{\widetilde}

\newcommand{\beqcd}[1]{\begin{equation*}\label{#1}\tag{#1}}
\newcommand{\eeqcd}{\end{equation*}}

\numberwithin{equation}{section}
\newenvironment{mylist}{
  \begin{enumerate}{}{%
      \setlength{\itemsep}{5pt} \setlength{\parsep}{0in}
      \setlength{\parskip}{0in} \setlength{\topsep}{0in}
      \setlength{\partopsep}{0in}
      \setlength{\leftmargin}{0.17in}}}{\end{enumerate}}

\def\parref#1{\ref{#1}}
\def\thmref#1{Theorem~\parref{#1}}
\def\propref#1{Prop.~\parref{#1}}
\def\corref#1{Cor.~\parref{#1}}     \def\remref#1{Remark~\parref{#1}}
\def\secref#1{\S\parref{#1}}

\def\lmref#1{Lemma~\parref{#1}}
\def\subsecref#1{\S\parref{#1}}

\def\makeop#1{\expandafter\def\csname#1\endcsname
  {\mathop{\rm #1}\nolimits}\ignorespaces}

\makeop{Hom}   \makeop{End}   \makeop{Aut}   
\makeop{Pic} \makeop{Gal}       \makeop{Div} \makeop{Lie}
\makeop{PGL}   \makeop{Corr} \makeop{PSL} \makeop{sgn} \makeop{Spf}
 \makeop{Tr} \makeop{Nr} \makeop{Fr} \makeop{disc}
\makeop{Proj} \makeop{supp} \makeop{ker}   \makeop{Im} \makeop{dom}
\makeop{coker} \makeop{Stab} \makeop{SO} \makeop{SL} \makeop{SL}
\makeop{Cl}    \makeop{cond} \makeop{Br} \makeop{inv} \makeop{rank}
\makeop{id}    \makeop{Fil} \makeop{Frac}  \makeop{GL} \makeop{SU}
\makeop{Trd}   \makeop{Sp} \makeop{Tr}    \makeop{Trd} \makeop{Res}
\makeop{ind} \makeop{depth} \makeop{Tr} \makeop{st} \makeop{Ad}
\makeop{Int} \makeop{tr}    \makeop{Sym} \makeop{can} \makeop{SO}
\makeop{torsion} \makeop{GSp} \makeop{Tor}\makeop{Ker} \makeop{rec}
\makeop{Ind} \makeop{Coker}
 \makeop{vol} \makeop{Ext} \makeop{gr} \makeop{ad}
 \makeop{Gr}\makeop{corank} \makeop{Ann}
\makeop{Hol} 
\makeop{Fitt} \makeop{Mp} \makeop{CAP}

\def\Spec{\mathrm{Spec}\,}

\DeclareMathAlphabet{\mathpzc}{OT1}{pzc}{m}{it}
\DeclareSymbolFont{cyrletters}{OT2}{wncyr}{m}{n}
\DeclareMathSymbol{\SHA}{\mathalpha}{cyrletters}{"58}

\def\makebb#1{\expandafter\def
  \csname bb#1\endcsname{{\mathbb{#1}}}\ignorespaces}
\def\makebf#1{\expandafter\def\csname bf#1\endcsname{{\bf
      #1}}\ignorespaces}
\def\makegr#1{\expandafter\def
  \csname gr#1\endcsname{{\mathfrak{#1}}}\ignorespaces}
\def\makescr#1{\expandafter\def
  \csname scr#1\endcsname{{\EuScript{#1}}}\ignorespaces}
\def\makecal#1{\expandafter\def\csname cal#1\endcsname{{\mathcal
      #1}}\ignorespaces}

\def\doLetters#1{#1A #1B #1C #1D #1E #1F #1G #1H #1I #1J #1K #1L #1M
                 #1N #1O #1P #1Q #1R #1S #1T #1U #1V #1W #1X #1Y #1Z}
\def\doletters#1{#1a #1b #1c #1d #1e #1f #1g #1h #1i #1j #1k #1l #1m
                 #1n #1o #1p #1q #1r #1s #1t #1u #1v #1w #1x #1y #1z}
\doLetters\makebb   \doLetters\makecal  \doLetters\makebf
\doLetters\makescr
\doletters\makebf   \doLetters\makegr   \doletters\makegr

\def\Gm{{\bbG}_{m}}

\normalsize

\makeop{Ram} \makeop{Rep} \makeop{mass}

\makeop{Bl}
\def\abs#1{\left|#1\right|}

\def\Fpbar{\bar{\mathbb F}_p}

\def\Qbarp{\C_p}
\def\Qp{\Q_p}
\def\Qbar{\bar\Q}
\def\Zbar{\bar{\Z}}
\def\Zbarp{\Zbar_p}

\def\Zp{\Z_p}

\def\rmN{{\mathrm N}}

\def\cA{{\mathcal A}}  

\def\cD{\mathcal D}
\def\cE{{\mathcal E}}
\def\cF{{\mathcal F}}  

\def\cI{\mathcal I}

\def\cK{{\mathcal K}}  
\def\cM{\mathcal M}
\def\cR{{\mathcal R}}
\def\cO{\mathcal O}
\def\cS{{\mathcal S}}
\def\cf{{\mathcal f}}
\def\cW{{\mathcal W}}

\def\cZ{\mathcal Z}

\def\cC{\mathcal C}

\def\cT{\mathcal T}

\def\cU{\mathcal U}

\def\EucA{{\EuScript A}}

\def\EucE{{\EuScript E}}

\def\EucL{{\EuScript L}}

\def\EucO{{\EuScript O}}

\def\bfc{\mathbf c}

\def\bfM{\mathbf M}

\def\bda{\mathbf a}
\def\bff{\mathbf f}
\def\bdh{\mathbf h}

\def\bfi{\mathbf i}

\def\bdw{\mathbf w}

\def\bdc{\mathbf c}

\def\bfdelta{\boldsymbol{\delta}}
\def\bftheta{\boldsymbol{\theta}}

\def\sL{\mathscr L}

\def\sO{\mathscr O}
\def\sS{\mathscr S}

\def\bbI{\mathbb I}

\newcommand{\Z}{\mathbf Z}
\newcommand{\Q}{\mathbf Q}
\newcommand{\R}{\mathbf R}
\newcommand{\C}{\mathbf C}
\newcommand{\A}{\mathbf A}    

\def\bbE{{\mathbb E}}

\def\bbmu{\boldsymbol{\mu}}

\def\fraka{{\mathfrak a}}
\def\frakb{{\mathfrak b}}
\def\frakc{{\mathfrak c}}

\def\frakq{\mathfrak q}

\def\frakm{\mathfrak m}

\def\frakl{\mathfrak l}

\def\frakF{{\mathfrak F}}
\def\frakE{\mathfrak E}

\def\frakD{\mathfrak D}

\def\frakC{{\mathfrak C}}

\def\frakX{\mathfrak X}

\def\frakN{\mathfrak N}
\def\il{\mathfrak i\frakl}

\def\bfone{{\mathbf 1}}

\def\ulA{\ul{A}}

\def\ulz{\ul{z}}

\def\ollam{\bar{\lam}}

\def\Zhat{\hat{\Z}}
\def\wbar{\bar{w}}

\def\wbar{\bar{w}}
\def\zbar{\bar{z}}

\def\ab{abelian variety }
\def\etale{{\'{e}tale }}

\def\padic{\text{$p$-adic }}

\def\BS{Bruhat-Schwartz }

\def\Neron{N\'{e}ron }

\newcommand{\<}{\langle}   
\renewcommand{\>}{\rangle} 

\def\isoto{\stackrel{\sim}{\to}}

\def\imply{\Rightarrow}
\def\ot{\otimes}

\def\hookto{\hookrightarrow}
\def\longto{\longrightarrow}
\def\ol{\overline}  \nc{\opp}{\mathrm{opp}} \nc{\ul}{\underline}

\newcommand{\pair}[2]{\< #1, #2\>}

\newcommand{\pairing}{\pair{\,}{\,}}

\def\XYmatrix{\xymatrix@M=8pt} 
\def\ncmd{\newcommand}
\ncmd{\xysubset}[1][r]{\ar@<-2.5pt>@{^(-}[#1]\ar@<2.5pt>@{_(-}[#1]}
\ncmd{\XYmatrixc}[1]{\vcenter{\XYmatrix{#1}}}
\ncmd{\xyto}[1][r]{\ar@{->}[#1]}
\ncmd{\xyinj}[1][r]{\ar@{^(->}[#1]}
\ncmd{\xysurj}[1][r]{\ar@{->>}[#1]}
\ncmd{\xyline}[1][r]{\ar@{-}[#1]}
\ncmd{\xydotsto}[1][r]{\ar@{.>}[#1]}
\ncmd{\xydots}[1][r]{\ar@{.}[#1]}
\ncmd{\xyleadsto}[1][r]{\ar@{~>}[#1]}
\ncmd{\xyeq}[1][r]{\ar@{=}[#1]} \ncmd{\xyequal}[1][r]{\ar@{=}[#1]}
\ncmd{\xyequals}[1][r]{\ar@{=}[#1]}
\ncmd{\xymapsto}[1][r]{l\ar@{|->}[#1]}\ncmd{\xyimplies}[1][r]{\ar@{=>}[#1]}
\ncmd{\xyiso}{\ar[r]_-{\sim}}
\def\injxy{\ar@{^(->}}

\newcommand{\MX}[4]{\begin{bmatrix}
{#1}& {#2}\\
{#3}&{#4}\end{bmatrix} }

\newcommand{\seesaw}[4]{{#1}\ar@{-}[rd]\ar@{-}[d]&{#2}\ar@{-}[d]\\
{#3}\ar@{-}[ru]&{#4}}

\def\ie{i.e. }

\def\cf{\mbox{{\it cf.} }}

\def\can{{can}}

\def\ENS{\mathfrak E\mathfrak n\mathfrak s}
\def\SCH{\mathfrak S\mathfrak c\mathfrak h}

\def\ch{{\mathbb I}}

\def\uf{\varpi} 
\def\Abs{{|\!\cdot\!|}} 

\def\Sg{{\varSigma}}  
\def\Sgbar{\Sg^c}
\def\ndivides{\nmid}
\def\ndivide{\nmid}
\def\x{{\times}}

\def\onehalf{{\frac{1}{2}}}
\def\al{\alpha}

\def\om{\omega}
\def\dirlim{\varinjlim}
\def\prolim{\varprojlim}
\def\iso{\simeq}
\def\con{\equiv}
\def\bksl{\backslash}
\newcommand\stt[1]{\left\{#1\right\}}
\def\ep{\epsilon}

\def\lam{\lambda}
\def\pii{\pi i}

\def\sg{\sigma}
\def\vp{\varphi}
\def\disjoint{\bigsqcup}
\def\bigot{\bigotimes}

\def\smid{\,|}

\def\dx{d^\x}

\def\AFf{\A_{\cF,f}}
\def\AKf{\A_{\cK,f}}
\def\AF{\A_\cF}
\def\AK{\A_\cK}
\def\setp{{(p)}}

\def\bbox{{(\Box)}}

\newcommand{\powerseries}[1]{\llbracket{#1}\rrbracket}
\newcommand{\formal}[1]{\widehat{#1}}
\renewcommand\pmod[1]{\,(\mbox{mod }{#1})}

\renewcommand\Re{\text{Re}\,}
\newcommand\Dmd[1]{\left<{#1}\right>} 
\def\vphi{\varphi}
\def\Cp{\C_p}

\def\alg{\mathrm{alg}} 
\def\OF{O}
\def\OK{R}

\def\adelef{\A_{\cF,f}}
\def\Section{\phi_{\ads,s,v}}
\def\OFv{O_v}
\def\OKv{R_v}
\def\uf{\varpi}

\def\wbar{\bar{w}}

\def\cmpt{\varsigma}
\def\cmptv{\varsigma_v}

\def\Csplit{\frakF}

\def\holES{\bbE^h_{\ads}}

\def\Cl{Cl}

\def\ENS{SETS}
\def\SCH{SCH}
\def\frakE{\mathfrak E}

\def\bR{\mathfrak O}
\def\baseR{\cW}
\def\frakN{N}
\def\ab{{\fraka,\frakb}}
\def\Gm{\mathbb G_m}

\def\wbar{\ol{w}}
\def\addchar{\psi}
\def\lp{\eta_p}
\def\ads{\lam}

\def\Beth{b}
\def\Eadsu{\bbE^h_{\ads,u}}
\def\smid{\,|}

\def\#{\sharp}
\title[The $\mu$-invariant of anticyclotomic $p$-adic $L$-functions]{
On the $\mu$-invariant of anticyclotomic $p$-adic $L$-functions
	for CM fields}
\author[M.-L. Hsieh]{Ming-Lun Hsieh}
\address{ Department of Mathematics~\\National Taiwan University ~ \\
No. 1, Sec. 4, Roosevelt Road, Taipei 10617, Taiwan~
}
\email{mlhsieh@math.ntu.edu.tw}
\date{August 12, 2012}
\subjclass[2010]{11F67 11R23}
\thanks{The author is partially supported by National Science Council grant 100-2115-M-002-012-}
\begin{document}
\begin{abstract}
In this article, we follow Hida's approach to study the $\mu$-invariant of the anticyclotomic projection of \padic Hecke $L$-function for CM fields along a branch character. We prove a conjecture of Gillard on the vanishing of the $\mu$-invariant and give a $\mu$-invariant formula for self-dual branch characters.
\end{abstract}
\maketitle
\tableofcontents
\def\ZZbox{\Z_{(\Box)\,}}
\def\cAbox{\cA^{(\Box)}_{K,\bdc}}
\def\cAboxn{\cA^{(\Box)}_{K,\bdc,n}}
\def\Zhatbox{\widehat{\Z}^{(\Box)}}
\def\qchKF{\tau_{\cK/\cF}}
\def\opcpt{K}
\def\sh{Sh}
\def\opn{K^n}
\def\lp{j}
\def\lpp{\eta^{(p)}}
\def\lsgN{\opcpt_1^n}
\def\Om{\boldsymbol{\omega}}
\def\wt{k}
\def\skewhf{\vartheta}
\def\Fv{F}
\def\Kv{E}
\def\nh{{n.h}}
\def\hatads{\wh\ads}
\def\f{\bff}
\def\OFp{\OF_p}
\newcommand\class[1]{{\left[#1\right]}}

\def\Ig{I}
\def\UF{{\delta_v}}
\def\bR{\EucO}
\def\CLKF{\Cl_-}
\def\holESc{\bbE^h_{\ads,\plideal}}
\def\HypNV{
\begin{enumerate}
\item[\rm{(R)}] The global root number $W(\nads)=1$, where $\nads:=\adsx\Abs^{-\onehalf}_{\AK}$,
\item[\rm{(L)}] $\mu_p(\adsx_v)=0$ for every inert $v|\frakC^-$.
\end{enumerate}
}
\def\HypMu{
\begin{enumerate}
\item[\rm{(R)}] the global root number $W(\nads)=1$, where $\nads:=\adsx\Abs^{-\onehalf}_{\AK}$.
\end{enumerate}
}
\def\Op{O_p}
\def\torsbgp{\cU_p}
\def\sO{\cO}
\def\Katzd{\theta}
\def\padicL{\EucL}
\def\Glv{Z(\frakC)}
\def\adsx{\chi}
\def\nads{\adsx^*}
\def\Imu{\mu_{\adsx,\Sg}}
\section*{Introduction}
The purpose of this article is to study the vanishing of Iwasawa $\mu$-invariant of anticyclotomic \padic Hecke $L$-functions for CM fields. To state our main result precisely, let us begin with some notation.
Let $p>2$ be an odd rational prime. Let $\cF$ be a totally real field of degree $d$ over $\Q$ and $\cK$ be a totally imaginary quadratic extension of $\cF$. Let $D_\cF$ be the discriminant of $\cF$. Fix two embeddings $\iota_\infty\colon\Qbar\to\C$ and $\iota_p\colon\Qbar\to\Qbarp$ once and for all. Let $c$ denote the complex conjugation on $\C$ which induces the unique non-trivial element of $\Gal(\cK/\cF)$. We assume the following hypothesis throughout this article:
\beqcd{ord}\text{Every prime of $\cF$ above $p$ splits in $\cK$.}\eeqcd
Fix a $p$-ordinary CM type $\Sg$, namely $\Sg$ is a CM type of $\cK$ such that \padic places induced by elements in $\Sg$ via $\iota_p$ are disjoint from those induced by elements in $\Sg c$. The existence of such a CM type $\Sg$ is assured by our assumption \eqref{ord}.
We recall some properties of \padic $L$-functions for CM fields. As in \cite{Katz:p_adic_L-function_CM_fields}, to a \Neron differential on an abelian scheme $\EucA_{/\Zbar}$ of CM type $(\cK,\Sg)$ we can attach the complex CM period $\Omega_\infty=(\Omega_{\infty,\sg})_{\sg}\in(\C^\x)^\Sg$ and the \padic CM period $\Omega_p=(\Omega_{p,\sg})_{\sg}\in(\Zbarp^\x)^\Sg$ .
Let $\frakC$ be a prime-to-$p$ integral ideal of $\cK$ and decompose $\frakC=\frakC^+\frakC^-$, where $\frakC^+$ (reps. $\frakC^-$) is a product of split prime factors (resp. ramified or inert prime factors) over $\cF$. Let $Z(\frakC)$ be the ray class group of $\cK$ modulo $\frakC p^\infty$. In \cite{Katz:p_adic_L-function_CM_fields} and \cite{HidaTilouine:KatzPadicL_ASENS}, a $\Zbarp$-valued \padic measure $\EucL_{\frakC,\Sg}$ on $Z(\frakC)$ is constructed such that
\begin{align*}\frac{1}{\Omega_p^{k\Sg+2\kappa}}\cdot\int_{Z(\frakC)}\wh\ads d\,\EucL_{\frakC,\Sg}=&L^{(p\frakC)}(0,\ads)\cdot Eul_p(\ads)Eul_{\frakC^+}(\ads)\\
&\times \frac{\pi^{\kappa}\Gamma_\Sg(k\Sg+\kappa)}{\sqrt{\abs{D_\cF}_\R}(\Im \skewhf)^\kappa\cdot \Omega_\infty^{k\Sg+2\kappa}}\cdot[\cO_\cK^\x:\cO_\cF^\x],
\end{align*}
where (i) $\ads$ is a Hecke character modulo $\frakC p^\infty$ of infinity type $k\Sg+\kappa(1-c)$ with either $k\geq 1$ and $\kappa\in\Z_{\geq 0}[\Sg]$ or $k\leq 1$ and $k\Sg+\kappa\in\Z_{>0}[\Sg]$, and $\wh\ads$ is the \padic avatar of $\ads$ regarded as a \padic Galois character via geometrically normalized reciprocity law, (ii)
$Eul_p(\ads)$ and $Eul_{\frakC^+}(\ads)$ are certain modified Euler factors (See \eqref{E:modifiedEuler.V}), (iii) $\skewhf$ is a well-chosen element in $\cK$ such that $c(\skewhf)=-\skewhf$.

We fix a Hecke character $\adsx$ of infinity type $k\Sg$ with $k\geq 1$ and suppose that
\[\text{$\frakC$ is the prime-to-$p$ conductor $\adsx$}.\] Let $\Gamma^-$ be the maximal $\Zp$-free quotient of the anticyclotomic quotient $Z(\frakC)^-$ of $Z(\frakC)$. Let $\EucL_{\adsx,\Sg}^-$ be the \padic measure on $\Gamma^-$ obtained by the pull-back of $\EucL_{\frakC,\Sg}$ along $\chi$. In other words, for every locally constant function $\phi$ on $\Gamma^-$, we have
\[\int_{\Gamma^-}\phi d\EucL_{\chi,\Sg}^-=\int_{Z(\frakC)}\phi \wh\chi d\,\EucL_{\frakC,\Sg}.\]
We call $\EucL^-_{\adsx,\Sg}$ the anticyclotomic \padic $L$-function with the branch character $\adsx$. Let $v_p$ be the valuation of $\Qbarp$ normalized so that $v_p(p)=1$.
Recall that the $\mu$-invariant $\mu(\vp)$ of a $\Zbarp$-valued \padic measure $\vp$ on a \padic group $H$ is defined to be
\[\mu(\vp)=\inf_{U\subset H\text{ open }} v_p(\vp(U)).\]
Let $\Imu^-:=\mu(\EucL^-_{\adsx,\Sg})$ be the $\mu$-invariant of $\EucL^-_{\adsx,\Sg}$. On the other hand, for each $v|\frakC^-$, we define the local invariant $\mu_p(\adsx_v)$ by
\[\mu_p(\adsx_v):=\inf_{x\in {\cK_v^\x}}v_p(\adsx_v(x)-1).\]
One of our main results in this paper is to give an exact formula of $\Imu^-$ when the Hecke character $\adsx$ is \emph{self-dual} in terms of the local invariants $\mu_p(\adsx_v)$ attached to $\adsx$. Recall that we say $\adsx$ is self-dual if $\adsx|_{\AF^\x}=\qchKF\Abs_{\AF}$, where $\qchKF$ is the quadratic character associated to $\cK/\cF$. It is not difficult to see that $\mu_p(\adsx_v)$ agrees with the one defined in \cite{Finis:vmu}\footnote{Self-dual characters are called \emph{anticyclotomic} therein.} when $\adsx$ is self-dual.

We remark that an important class of self-dual characters are those associated to CM abelian varieties over
totally real fields (\cf \cite[20.15]{Shimura:ABV-with-CM}). Our first result is the determination of $\Imu^-$ if $\adsx$ is self-dual.
\begin{thmA}\label{T:A.V}
Suppose that $p\ndivide D_{\cF}$. Let $\adsx$ be a self-dual Hecke character of $\cK^\x$ such that
\HypMu
Then we have
\[\Imu^-=\sum_{v|\frakC^-}\mu_p(\adsx_v).\]
\end{thmA}
If the branch character $\adsx$ is not self-dual, we do not get the precise formula of $\Imu^-$, but we can still offer the following criterion on the vanishing of $\Imu^-$ at least when $\adsx$ is not residually self-dual (\corref{C:2.V}).
\begin{thmA}\label{T:B.V}Suppose that $p\ndivide  D_{\cF}$ and that
\begin{itemize}
\item[(L)] $\mu_p(\adsx_v)=0$ for every $v|\frakC^-$,
\item[(N)]$\adsx$ is not residually self-dual, namely $\wh{\adsx}_+\not\con \qchKF\om_\cF\pmod{\frakm}$.
\end{itemize}
Then $\Imu^-=0$.
\end{thmA}
The above two theorems verify a conjecture of Gillard \cite[p.21 Conjecture (ii)]{Gillard:Remark_CM} when $p\ndivide  D_{\cF}$ (\cf the discussion in \cite[p.3]{Hida:VVmu}). Note that by the functional equation of complex $L$-functions, the $\mu$-invariant $\Imu^-=\infty$ (\ie $\EucL^-_{\adsx,\Sg}=0$) if $\adsx$ is self-dual and $W(\nads)=-1$. If $\cK$ is an imaginary quadratic field, \thmref{T:A.V} is proved by T. Finis \cite{Finis:vmu}. For general CM fields, both theorems are proved by Hida \cite{Hida:mu_invariant} under the assumption that $\frakC^-=(1)$. The idea of Hida is to construct a family of $p$-integral Eisenstein series $\stt{\cE^\circ_a}_{a\in\bfD}$ indexed by a suitable finite subset $\bfD$ of transcendental automorphism groups of the deformation space of the ordinary CM abelian variety $\EucA{}_{/\Fpbar}$ such that the $t$-expansion of some linear combination $\cE$ of $\stt{\cE^\circ_a}_{a\in\bfD}$ at the CM point $\EucA{}_{/\Fpbar}$ gives rise to the power series expansion of the measure $\EucL^-_{\adsx,\Sg}$. Using a key result on the linear independence of modular forms modulo $p$ \cite[Thm.\,3.20, Cor.\,3.21]{Hida:mu_invariant} combined with the $q$-expansion principle, Hida reduces the determination of $\Imu^-$ to an explicit computation of the Fourier coefficients of the Eisenstein series $\cE^\circ_a$. Assuming $\frakC^-=(1)$, Hida computes the Fourier coefficients of $\cE^\circ_a$, from which he is able to deduce a necessary and sufficient condition for the vanishing of $\mu$-invariant $\Imu^-$. As remarked by Hida, the reason for the assumption $\frakC^-=(1)$ is that the calculation of the Fourier coefficients is rather complicated if $\frakC^-\not =(1)$.

The aim of this paper is to lift the assumption $\frakC^-=(1)$. The idea is to construct a new family of the \emph{toric} Eisenstein series $\stt{\cE_a}_{a\in\bfD}$ of which the Fourier coefficients can be computed with the help of representation theory and rewrite $\cE$ as a linear combination of these $\cE_a$. The construction of $\cE_a$ relies on a special choice of local sections in a certain local principal series at each place of $\cF$. The choice of local sections outside $p$ has been made in \cite[\S\,4.3]{Hsieh:Hecke_CM}. At the places above $p$, such a choice is inspired by \cite{HLS}, where Katz \padic Eisenstein measure is studied from representation theoretic point of view. To obtain the formula of $\Imu^-$, we have to compute explicitly all Fourier coefficients of $\cE_a$, which in turn can be decomposed into a product of the local Whittaker integrals attached to these local sections. In \cite{Hsieh:Hecke_CM}, the local Whittaker integrals are determined explicitly by a straightforward computation at all places $v$ other than those inert or ramified with $v(\frakC^-)>1$. In general, for each $v|\frakC^-$ and $\beta\in\cF^\x_v$, the local $\beta$-th Whittaker integral is essentially the partial Gauss sum $A_{\beta}(\adsx_v)$ given by
\[A_{\beta}(\adsx_v)=\int_{\cF_v}\adsx_v^{-1}(x+2^{-1}\delta)\addchar^\circ(\beta x)dx,\]
where $\delta\in\cK_v$ such that $c(\delta)=-\delta$ and $\addchar^\circ$ is an additive character on $\cF_v$. It turns out that the $\mu$-invariant $\Imu^-$ is determined by the \padic valuations of $A_\beta(\adsx_v)$ with $\beta$ in the global field $\cF$ for all $v|\frakC^-$, and in particular, the non-vanishing modulo $p$ of $A_\beta(\adsx_v)$ for some $\beta\in\cF$ implies the vanishing of $\Imu^-$. It seems that $A_\beta(\adsx_v)$ is difficult to evaluate in general. However, we can deduce the vanishing of $\Imu^-$, assuming the vanishing of the local invariant $\mu_p(\adsx_v)$ for each $v|\frakC^-$. In other words, we can show the existence of $\beta\in\cF$ such that
\beqcd{n.v.}A_{\beta}(\adsx_v)\not\con 0\pmod{\frakm}\text{ for all }v|\frakC^-.\eeqcd Indeed, it is shown in \cite[Lemma 6.4]{Hsieh:Hecke_CM} that at each $v|\frakC^-$ there exists some $\beta_v$ in the local field $\cF_v$ with $A_{\beta_v}(\adsx_v)$ non-vanishing modulo $p$, and then the strong approximation enables us to deduce easily the existence of $\beta$ in the global field $\cF$ with the property \eqref{n.v.} if $\chi$ is not residually self-dual. In the special case $\adsx$ is self-dual and the global root number $W(\nads)=+1$, we further need to show that this $\beta_v$ satisfies certain epsilon dichotomy (See \propref{P:NVAbeta.V}). Under the assumption the ramified part of $\frakC^-$ is square-free, this epsilon dichotomy for $\beta_v$ is verified in \cite{Hsieh:Hecke_CM}. To treat the general case, we identify $A_\beta(\adsx_v)$ with the Whittaker integral associated to a certain Siegel-Weil section in the degenerate principal series of $U(1,1)$ and apply results in \cite[\S 6 and \S 8]{Harris:Theta_dichotomy} to show that $\beta$ indeed satisfies the epsilon dichotomy whenever $A_\beta(\adsx_v)\not =0$.

This paper is organized as follows. In the first three sections, we review the theory of \padic Hilbert modular forms and CM points in Hilbert modular varieties.
In \secref{S:ES}, we give the construction of our \padic Eisenstein measure $\cE$ (\propref{P:ESmeasure.V}). We show in \propref{P:EVCM} that the period integral of $\cE$ against a non-split torus gives rise to \padic $L$-functions for CM fields constructed in \cite{Katz:p_adic_L-function_CM_fields} and \cite{HidaTilouine:KatzPadicL_ASENS}.  In \secref{S:Hida.V}, we review Hida's theorem on the linear independence of modular forms applied by the automorphisms in $\bfD$ proved in \cite{Hida:mu_invariant}. Finally, in \secref{S:ThetaIntegral.V} after establishing a crucial lemma (\lmref{L:key.V}) relating the non-vanishing of $A_\beta(\adsx_v)$ to the epsilon dichotomy of $\beta$, we prove our main result (\thmref{T:main.V}).

\begin{thank}
The author would like to thank Professor Tamotsu Ikeda for helpful suggestions. The author is also very grateful to the referee for the careful reading of the manuscript and the suggestions on the improvements of the exposition.
\end{thank}
\def\lpp{\ol{\eta}^\setp}
\def\lsgNN{\opcpt_1(p^\infty)}
\def\plideal{\frakc}
\def\setpN{{(p\frakN)}}
\def\Zbarsetp{\Zbar_\setp}

\section{Notation and definitions}\label{S:Notation}
\subsection{}Let $\cF$ be a totally real field of degree $d$ over $\Q$ and let $\cK$ be a totally imaginary quadratic extension of $\cF$. Let $c$
be the complex conjugation, the unique non-trivial
element in $\Gal(\cK/\cF)$. Let $\OF$ and $\OK$ be the ring of integer of $\cF$ and $\cK$ respectively. Let $\cD_\cF$ (resp. $D_\cF$) be the different (resp. discriminant) of $\cF/\Q$. Let $\cD_{\cK/\cF}$ (resp. $D_{\cK/\cF})$ be the different (resp. discriminant) of $\cK/\cF$. For every fractional ideal $\frakb$ of $\OF$, set $\frakb^*=\frakb^{-1}\cD_\cF^{-1}$. Denote by $\bda=\Hom(\cF,\C)$ the set of archimedean places of $\cF$.
Denote by $\bdh$ (resp. $\bdh_\cK$) the set of finite places of $\cF$ (resp. $\cK$).
We often write $v$ for a place of $\cF$ and $w$ for the place of $\cK$ above $v$. Denote by $\cF_v$ the completion of $\cF$ at $v$ and by $\uf_v$ a uniformizer of $\cF_v$.
Let $\cK_v=\cF_v\ot_\cF\cK$.

We fix a rational prime $p$. Throughout this article, in addition to \eqref{ord}, we further assume
\beqcd{unr}2<p\ndivide D_\cF.\eeqcd
If $\Sg$ is a CM type of $\cK$, we put \[\Sg_p=\stt{w\in \bdh_\cK\mid w|p\text{ and $w$ is induced by $\iota_p\circ\sg$ for $\sg\in\Sg$}}.\]
Recall that $\Sg$ is $p$-ordinary if $\Sg_p\cap\Sg_pc=\emptyset$ and $\Sg_p\cup\Sg_pc$ is the set of places of $\cK$ lying above $p$. The existence of $p$-ordinary CM types is assured by \eqref{ord}. Hereafter
we fix a $p$-ordinary CM type $\Sg$, and identify
$\Sg$ with $\bda$ by the restriction to $\cF$.
\subsection{}
If $L$ is a number field, $\A_L$ is the adele of $L$ and $\A_{L,f}$
is the finite part of $\A_L$. The ring of integers of
$L$ is denoted by $\cO_L$. For $a\in\A_L$, we put
\[\il_L(a):=a(\cO_L\ot\Zhat)\cap L.\]
Denote by $G_L$ the absolute Galois group and by $\rec_L:\A_L^\x\to G^{ab}_{L}$ the geometrically normalized reciprocity law. 
Let $\addchar_\Q$ be the standard additive character
of $\A_\Q/\Q$ such that $\addchar_\Q(x_\infty)=\exp(2\pii x_\infty),\,x_\infty\in\R$. We define $\addchar_L:\A_L/L\to\C^\x$ by
$\addchar_L(x)=\addchar_\Q\circ\Tr_{L/\Q}(x)$. For $\beta\in L$, $\addchar_{L,\beta}(x)=\addchar_L(\beta x)$. If $L=\cF$, we write $\addchar$ for $\addchar_\cF$.

We choose once and for all an embedding
$\iota_\infty:\Qbar\hookto\C$ and an isomorphism
$\iota:\C\iso\Qbarp$, where $\Qbarp$ is the completion of an
algebraic closure of $\Qp$. Let
$\iota_p=\iota\iota_\infty:\Qbar\hookto\Qbarp$ be their composition.
We regard $L$ as a subfield in $\C$ (resp. $\C_p$) via
$\iota_\infty$ (resp. $\iota_p$) and $\Hom(L,\Qbar)=\Hom(L,\C_p)$.

Let $\Zbar$ be the ring of algebraic integers of $\Qbar$ and let $\Zbarp$ be the \padic completion of $\Zbar$ in $\Qbarp$. Let $\Zbar$ be the ring of algebraic integers of $\Qbar$ and let $\Zbarp$ be the \padic completion of $\Zbar$ in $\Qbarp$ with the maximal ideal $\frakm_p$. Let $\frakm=\iota_p^{-1}(\frakm_p)$.
\subsection{}
Let $F$ be a local field. We fix the choice of our Haar measure $dx$ on $F$. If $F$ is archimedean, $dx$ is the Lebesgue measure on $F$. If $F$ is a non-archimedean local field, $dx$ (resp. $\dx x$) is the Haar measure on $F$ (resp. $F^\x$) normalized so that $\vol(\cO_F,dx)=1$ (resp. $\vol(\cO_F^\x,\dx x)=1$). Denote by $\Abs_F$ the absolute value of $F$ such that $d(ax)=\abs{a}_Fdx$ for $a\in F^\x$. We often drop the subscript $F$ if it is clear from the context.

\section{Hilbert modular Shimura varieties and Hilbert modular forms}\label{S:Hilbert}
\subsection{}The purpose of this section is to review standard facts about Hilbert modular Shimura varieties and Hilbert modular forms. We follow the exposition in \cite[\S 4.2]{Hida:p-adic-automorphic-forms}. Let $V=\cF e_1\oplus\cF e_2$ be a two dimensional
$\cF$-vector space and $\pairing:V\x V\to \cF$ be the $\cF$-bilinear
alternating pairing defined by $\pair{e_1}{e_2}=1$. Let $\sL=\OF
e_1\oplus \OF^* e_2$ be the standard $\OF$-lattice in $V$. Let $G=\GL_2{}_{/\cF}$. For $g=\MX{a}{b}{c}{d}\in M_2(\cF)$, we define an involution $g'=\MX{d}{-b}{-c}{a}$. If $g\in\GL_2(\cF)=G(\cF)$, then $g'=g^{-1}\det g$. We
identify vectors in $V$ with row vectors according to the basis $e_1,e_2$, so $G$ has a natural right
action on $V$. Define a left action of $G$ on $V$ by $g*x:=x\cdot g',\,x\in V$.

For a finite place $v$ of $\cF$, we put
\[\opcpt^0_v=\stt{g\in G(\cF_v)\mid g*(\sL\ot_{\OF}\OFv)=\sL\ot_{\OF}\OFv}.\]
Let $\opcpt^0=\prod_{v\in\bdh}\opcpt^0_v$ and $\opcpt^0_p=\prod_{v|p}\opcpt^0_v$.
For a prime-to-$p$ positive integer $\frakN$, we define an open-compact subgroup $U(N)$ of $G(\AFf)$ by
\beq\label{E:opcpt.N}U(\frakN):=\stt{g\in G(\AFf)\mid g\con 1\pmod{\frakN\sL}}.\eeq
 Let $\opcpt$ be an open-compact subgroup of $G(\AFf)$ such that $K_p=\opcpt^0_p$. We assume that $\opcpt\supset U(\frakN)$ for some $\frakN$ as above and that
 $\opcpt$ is sufficiently small so that the following condition holds:
\beqcd{neat}\opcpt\text{ is neat and }\det (\opcpt)\cap
\OF_+^\x\subset (K\cap \OF^\x)^2.\eeqcd

\def\OPU{U}
\subsection{Kottwitz models}
We recall Kottwitz models of Hilbert modular Shimura varieties.
\begin{defn}[$S$-quadruples]\label{D:6.H} Let $\Box$ be a finite set of
rational primes and let $U$ be an open-compact subgroup of $\opcpt^0$ such that $\OPU\supset U(\frakN)$ for some positive integer $\frakN$ prime to $\Box$.
Let $\baseR_U=\Z_\bbox[\zeta_{\frakN}]$ with $\zeta_{\frakN}=\exp(\frac{2\pii}{\frakN})$. Define the fibered category $\cA^\bbox_{\OPU}$ over the category $SCH_{/\baseR_U}$ of schemes over $\baseR_U$ as follows. Let $S$ be a locally noethoerian connected $\baseR_U$-scheme and let $\ol{s}$ be a geometric point of $S$. Objects are abelian varieties with real
multiplication (AVRM) over $S$ of level $\OPU$, \ie a
$S$-\emph{quadruple} $(A,\ollam,\iota,\ol{\eta}^\bbox)_S$ consisting of the
following data:
\begin{enumerate}
\item $A$ is an abelian scheme of dimension $d$ over $S$.
\item $\iota :\OF\hookrightarrow \End_S A\ot_\Z\ZZbox$.
\item $\lam$ is a prime-to-$\Box$ polarization of $A$ over $S$ and
$\ollam$ is the $\OF_{\bbox,+}$-orbit of $\lam$. Namely
\[\ollam=\OF_{\bbox,+}\lam:=\stt{\lam'\in\Hom(A,A^t)\ot_\Z\ZZbox\mid \lam'=\lam\circ a,\,a\in O_{\bbox,+}}.\]
\item $\ol{\eta}^\bbox=\eta^\bbox\OPU^\bbox$ is a $\pi_1(S,\ol{s})$-invariant $\OPU^\bbox$-orbit of isomorphisms of $\cO_\cK$-modules $\eta^\bbox: \sL\ot_\Z\A_f^\bbox\isoto V^\bbox(A_{\ol{s}}):=H_1(A_{\ol{s}},\A_f^\bbox)$. Here we define $\eta^\bbox g$ for $g\in G(\AFf^\bbox)$ by $\eta^\bbox g(x)=\eta^\bbox(g*x)$.
\end{enumerate}
Furthermore, $(A,\ollam,\iota,\ol{\eta}^\bbox )_S$ satisfies the
following conditions:
\begin{itemize}
\item Let ${}^t$ denote the Rosati involution induced by $\lam$ on
$\End_SA\ot\ZZbox$. Then $\iota(b)^t=\iota(b),\, \forall\, b\in
\OF.$
\item Let $e^\lam$ be the Weil pairing induced by $\lam$. Lifting the isomorphism $\Z/\frakN\Z\iso \Z/\frakN\Z(1)$ induced by $\zeta_{\frakN}$ to an isomorphism $\zeta:\Zhat\iso\Zhat(1)$, we can regard $e^\lam$ as an $\cF$-alternating form
$e^\lam:V^\bbox(A)\times V^\bbox(A)\to O^*\ot_\Z\A_f^\bbox$. Let $e^\eta$ denote the
$\cF$-alternating form on $V^\bbox(A)$ induced by
$e^\eta(x,x')=\pair{x\eta}{x'\eta}$. Then
\[e^\lam=u\cdot e^\eta\text{ for some }u\in\AFf^\bbox.\]
\item As
$\OF\ot_\Z\cO_S$-modules, we have an isomorphism $\Lie A\iso \OF\ot_\Z\cO_S$ locally under Zariski topology of $S$.
\end{itemize}
For two $S$-quadruples $\ulA=(A,\ollam,\iota,\ol{\eta}^\bbox )_S$ and $\ul{A'}=(A',\ol{\lam'},\iota',\ol{\eta'}{}^\bbox)_S$,
we define morphisms by
\[\Hom_{\cA^\bbox_{\opcpt}}(\ulA,\ul{A'})=\stt{\phi\in \Hom_{\OF}(A,A')\mid
\phi^*\ol{\lam'}=\ollam,\,\phi\circ \ol{\eta'}{}^\bbox=\ol{\eta}^\bbox }.\] We
say $\ulA\sim \ul{A'}$ (resp. $\ulA\iso \ul{A'}$) if there exists a
prime-to-$\Box$ isogeny (resp. isomorphism) in
$\Hom_{\cA_\opcpt^\bbox}(\ulA,\ul{A'})$.
\end{defn}
We consider the cases when $\Box=\emptyset$ and $\stt{p}$. When
$\Box=\emptyset$ is the empty set and $\OPU$ is an open-compact
subgroup in $G(\AFf^\bbox)=G(\AFf)$, we define the functor
$\cE_{\OPU}:\SCH_{/\baseR_U}\to\ENS$ by
\[\cE_\OPU(S)=\stt{(A,\ollam,\iota,\ol{\eta} )_S\in\cA_\opcpt(S)}/\sim.\] By the theory of Shimura-Deligne, $\cE_{\OPU}$ is represented by
$\sh_\OPU$ which is a quasi-projective scheme over $\baseR_{\OPU}$. We define
the functor $\frakE_\OPU:\SCH_{/\baseR_{\OPU}}\to\ENS$ by
\[\frakE_\OPU(S)=\stt{(A,\ollam,\iota,\ol{\eta})\in\cA^\bbox_\OPU(S)\mid \eta^\bbox(\sL\ot_\Z\Zhat)=H_1(A_{\ol{s}},\Zhat)}/\iso.\]
By the discussion in \cite[p.136]{Hida:p-adic-automorphic-forms}, we have $\frakE_\opcpt\isoto\cE_\opcpt$ under the hypothesis \eqref{neat}.

When $\Box=\stt{p}$ and $\OPU=\opcpt$, we let $\baseR=\baseR_\opcpt=\Z_\setp[\zeta_N]$ and define functor
$\cE^\setp_{\opcpt}:\SCH_{/\baseR}\to\ENS$ by
\[\cE^\setp_{\opcpt}(S)=\stt{(A,\ollam,\iota,\ol{\eta}^\setp)_S\in\cA_{\opcpt^\setp}^\setp(S)}/\sim.\]
In \cite{Kottwitz:Points-On-Shimura-Varieties}, Kottwitz shows
$\cE^\setp_{\opcpt}$ is representable by a quasi-projective
scheme $\sh^\setp_{\opcpt}$ over $\baseR$ if $\opcpt$ is neat. Similarly we
define the functor $\frakE_K^\setp:\SCH_{/\baseR}\to\ENS$ by
\[\frakE_K^\setp(S)=\stt{(A,\ollam,\iota,\ol{\eta}^\setp)\in\cA^\setp_\opcpt(S)\mid \eta^\setp(\sL\ot_\Z\Zhat^\setp)=H_1(A_{\ol{s}},\Zhat^\setp)}/\iso.\]
It is shown in \cite[\S 4.2.1]{Hida:p-adic-automorphic-forms} that
$\frakE^\setp_K\isoto\cE^\setp_K$.

Let $\plideal$ be a prime-to-$p\frakN$ ideal of $\OF$ and let $\bfc\in (\AFf^\setpN)^\x$ such that $\plideal=\il_{\cF}(\bfc)$. We say $(A,\lam,\iota,\ol{\eta}^\setp)$ is $\plideal$-polarized if $\lam\in\ollam$ such that $e^\lam=ue^{\eta},\,u\in \bfc\det(\opcpt)$.
The isomorphism class $[(A,\lam,\iota,\ol{\eta}^\setp)]$ is independent of a choice of $\lam$ in $\ol{\lam}$ under the assumption \eqref{neat} (\cf\cite[p.136]{Hida:p-adic-automorphic-forms}).
We consider the functor
\[\frakE^\setp_{\plideal,\opcpt}(S)=\stt{\text{$\plideal$-polarized $S$-quadruple }[(A,\lam,\iota,\ol{\eta}^\setp)_S]\in\frakE^\setp_{\opcpt}(S)}.\]
Then $\frakE^\setp_{\plideal,\opcpt}$ is represented by a geometrically irreducible scheme $\sh^\setp_\opcpt(\plideal)_{/\baseR}$, and we have
\beq\label{E:decomposition}\sh^\setp_\opcpt{}_{/\baseR}=\disjoint_{[\plideal]\in\Cl^+_\cF(\opcpt)}\sh^\setp_\opcpt(\plideal)_{/\baseR},\eeq
where $\Cl^+_\cF(\opcpt)$ is the narrow ray class group of $\cF$ with level $\det(\opcpt)$.
\def\Igusa{\Ig_{\opcpt,n}}
\subsection{Igusa schemes}\label{17.H}
Let $n$ be a positive integer. Define the functor $\cI^\setp_{K,n}:\SCH_{/\baseR}\to\ENS$ by
\[S\mapsto \cI^\setp_{K,n}(S)=\stt{(A,\ollam,\iota,\lpp,\lp)_S}/\sim,
\]
where $(A,\ollam,\iota,\lpp)_S$ is a $S$-quadruple, $\lp$ is a level $p^n$-structure, \ie an $\OF$-group scheme morphism:
\[\lp:\OF^*\ot_\Z\bbmu_{p^n}\hookto A[p^n],\]
and $\sim$ means modulo prime-to-$p$ isogeny. It is known that $\cI^\setp_{K,n}$ is relatively representable over $\cE^\setp_{K}$ (\cf\cite[Lemma (2.1.6.4)]{HLS}) and thus is represented by a scheme $\Igusa$.

Now we consider $S$-quintuples $(A,\lam,\iota,\lpp,\lp)_S$ such that $[(A,\lam,\iota,\lpp)]\in\frakE^\setp_{\plideal,K}(S)$. Define the functor $\cI^\setp_{K,n}(\plideal):\SCH_{/\baseR}\to\ENS$ by
\begin{align*}S\mapsto\cI^\setp_{K,n}(\plideal)(S)&=\stt{(A,\lam,\iota,\lpp,\lp)_S\text{ as above}}/\iso.\\
\end{align*}
Then $\cI^\setp_{K,n}(\plideal)$ is represented by a scheme $\Igusa(\plideal)$ over $\sh^\setp_\opcpt(\plideal)$, and $\Igusa(\plideal)$ can be identified with a geometrically irreducible subscheme of $\Igusa$ (\cite[Thm.\,(4.5)]{DR_padic_L}).
For $n\geq n'>0$, the natural morphism
$\pi_{n,n'}:\Igusa(\plideal)\to\Ig_{\opcpt,n'}(\plideal)$ induced by the
inclusion $\OF^*\ot\bbmu_{p^{n'}}\hookto \OF^*\ot\bbmu_{p^n}$ is finite
\etale.
The forgetful
morphism $\pi:\Igusa(\plideal)\to \sh^\setp_{\opcpt}(\plideal)$ defined by
$\pi:(\ulA,\lp)\mapsto \ulA$ is \etale for all $n>0$. Hence
$\Igusa(\plideal)$ is smooth over $\Spec\baseR$. We write $\Ig_K(\plideal)$ for $\prolim_n\Igusa(\plideal)$.

\subsection{Complex uniformization}\label{S:cpx} We describe the complex points $\sh_\OPU(\C)$ for $U\subset G(\AFf)$. Put
\[X^+=\stt{\tau=(\tau_\sg)_{\sg\in\bda}\in\C^{\bda}\mid \Im \tau_\sg >0\text{ for all } \sg\in\bda}.\]
Let $\cF_+$ be the set of totally positive elements in $\cF$ and let $G(\cF)^+=\stt{g\in G(\cF)\mid \det g\in\cF_+}$. Define the complex Hilbert modular Shimura variety by
\[M(X^+,\OPU):=G(\cF)^+\bksl X^+\x G(\AFf)/\OPU.\]
It is well known that $M(X^+,\opcpt)\isoto\sh_\OPU(\C)$ by the theory
of abelian varieties over $\C$ (\cf\cite[\S\,4.2]{Hida:p-adic-automorphic-forms}).

For $\tau=(\tau_\sg)_{\sg\in\bda}\in X^+$, we let $p_\tau$ be the isomorphism $V\ot_\Q\R\isoto \C^{\bda}$ defined by
$p_\tau(ae_1+be_2)=a\tau+b$ with $a,b\in \cF\ot_\Q\R=\R^{\bda}$. We can associate a AVRM to $(\tau,g)\in X^+\x G(\AFf)$ as follows.
\begin{itemize}
\item The complex abelian variety $\EucA_g(\tau)=\C^{\bda}/p_\tau(g*\sL)$.
\item  The $\cF_+$-orbit of polarization
$\ol{\pairing}_\can$ on $\EucA_g(\tau)$ is given by the Riemann form $\pairing_\can:=\pairing\circ p_\tau^{-1}$.
\item The $\iota_\C:O\hookto\End \EucA_g(\tau)\ot_\Z\Q$ is induced from the pull back of the natural $\cF$-action on $V$ via $p_\tau$.
\item The level structure $\eta_g:
\sL\ot_\Z\A_f \isoto (g*\sL)\ot_\Z\A_f=H_1(\EucA_g(\tau),\A_f)$ is defined by  $\eta_g(v)= g*v$.\end{itemize}
Let $\ul{\EucA_g(\tau)}$ denote the $\C$-quadruple $(\EucA_g(\tau),\ol{\pairing}_\can,\iota_\C,\opcpt\eta_g)$. Then the map $[(\tau,g)]\mapsto [\ul{\EucA_g(\tau)}]$ gives rise to
an isomorphism $M(X^+,\OPU)\isoto \sh_\OPU(\C)$.

For a positive integer $n$, the exponential map gives the isomorphism $\exp(2\pii -):p^{-n}\Z/\Z\iso\bbmu_{p^n}$ and thus induces a level $p^n$-structure $\lp(g_p)$:
\begin{align*}\lp(g_p)\colon&O^*\ot_\Z\bbmu_{p^n}\isoto O^*e_2\ot_\Z p^{-n}\Z/\Z\hookto \sL\ot_\Z p^{-n}\Z/\Z\stackrel{g*}\isoto \EucA_g(\tau)[p^n].
\end{align*}
Put
\[\lsgN:=\stt{g\in\opcpt\mid g_p\con\MX{1}{*}{0}{1}\pmod{p^n}}.\]
We have a non-canonical isomorphism:
\begin{align*}
M(X^+,\lsgN)&\isoto \Igusa(\C)\\
 [(\tau,g)]&\mapsto  [(\EucA_g(\tau),\ol{\pairing}_\can,\iota_\C,\ol{\eta}^\setp_g, \lp(g_p))].\end{align*}

Let $\ulz=\stt{z_\sg}_{\sg\in\bda}$ be the standard complex coordinates of $\C^{\bda}$ and $d\ulz=\stt{dz_\sg}_{\sg\in\bda}$. Then $O$-action on $d\ulz$ is given by
$\iota_\C(\al)^* dz_\sg=\sg(\al)dz_\sg,\,\sg\in\bda=\Hom(\cF,\C)$. Let $z=z_{id}$ be the coordinate corresponding to $\iota_\infty:\cF\hookto\Qbar\hookto\C$.
Then
\beq\label{E:7.N}(\OF\ot_\Z\C) dz=H^0(\EucA_g(\tau),\Omega_{\EucA_g(\tau)/\C}).\eeq
\subsection{Hilbert modular forms}For $\tau\in \C$ and $g=\MX{a}{b}{c}{d}\in\GL_2(\R)$,
we put \beq\label{E:4.N}J(g,\tau)=c\tau+d.\eeq For
$\tau=(\tau_\sg)_{\sg\in\bda}\in X^+$ and $g_\infty=(g_\sg)_{\sg\in\bda}\in
G(\cF\ot_\Q\R)$, we put
\[\ul{J}(g_\infty,\tau)=\prod_{\sg\in\bda}J(g_\sg,\tau_\sg).\]
\begin{defn}
Denote by $\bfM_k(\lsgN,\C)$ the space of holomorphic Hilbert modular form of parallel weight $k\Sg$ and level $\lsgN$.
Each $\f\in\bfM_k(\lsgN,\C)$ is a $\C$-valued function $\f:X^+\x G(\AFf)\to \C$ such that the function $\f(-,g_f):X^+\to\C$ is holomorphic for each $g_f\in G(\AFf)$ and
\[\f(\al(\tau,g_f)u)=\ul{J}(\al,\tau)^{k\Sg}\f(\tau,g_f)\text{ for all }u\in \lsgN\text{ and }\al\in G(\cF)^+.\]
\end{defn}
For every $\f\in \bfM_k(\lsgN,\C)$, we have the Fourier expansion
\[\f(\tau,g_f)=\sum_{\beta\in\cF_+\cup\stt{0}}W_\beta(\f,g_f)e^{2\pii\Tr_{\cF/\Q}(\beta \tau)}.\]
We call $W_\beta(\f,g_f)$ the $\beta$-th Fourier coefficient of
$\f$ at $g_f$.

For a semi-group $L$ in $\cF$, let $L_+=\cF_+\cap L$ and $L_{\geq 0}=L_+\cup \stt{0}$. If $B$ is a ring, we denote by $B\powerseries{L}$ the set of all formal series
\[\sum_{\beta\in L}a_\beta q^\beta,\,a_\beta\in B.\]
Let $a,b\in(\AFf^{(p\frakN)})^\x$ and let $\fraka=\il_\cF(a)$ and
$\frakb=\il_\cF(b)$. The $q$-expansion of $\f$ at the cusp $(\fraka,\frakb)$ is given by
\beq\label{E:FC0}\f|_{(\fraka,\frakb)}(q)=\sum_{\beta\in (N^{-1}\fraka\frakb)_{\geq 0}}W_\beta(\f,\MX{a^{-1}}{0}{0}{b})q^\beta\in \C\powerseries{(N^{-1}\fraka\frakb)_{\geq 0}}. \eeq
If $B$ is a $\baseR$-algebra in $\C$, we put
\begin{align*}
\bfM_k(\plideal,\lsgN,B)&=\stt{\f\in \bfM_k(\lsgN,\C)\mid \f|_{(\fraka,\frakb)}(q)\in B\powerseries{(N^{-1}\fraka\frakb)_{\geq 0}}\text{ for all }(\fraka,\frakb)
\text{ with $\fraka\frakb^{-1}=\frakc$}}.\\
\end{align*}
\subsubsection{Tate objects}
Let $\sS$ be a set of $d$ linearly $\Q$-independent elements in $\Hom(\cF,\Q)$ such that $l(\cF_+)>0$ for $l\in\sS$. If $L$ is a lattice in $\cF$ and $n$ a positive integer, let
$L_{\sS,n}=\stt{x\in L\mid l(x)>-n\text{ for all }l\in\sS}$ and put $B((L;\sS))=\lim\limits_{n\to\infty} B\powerseries{L_{\sS,n}}$.
To a pair $(\fraka,\frakb)$ of two prime-to-$pN$ fractional ideals , we can attach the Tate AVRM $Tate_{\fraka,\frakb}(q)=\fraka^*\ot_\Z\Gm/q^{\frakb}$ over $\Z((\fraka\frakb;\sS))$ with $O$-action $\iota_\can$. As described in \cite{Katz:p_adic_L-function_CM_fields}, $Tate_{\fraka,\frakb}(q)$ has a canonical $\fraka\frakb^{-1}$-polarization $\lam_\can$ and also carries $\Om_\can$ a canonical $\OF\ot\Z((\fraka\frakb;\sS))$-generator of $\Omega_{Tate_{\fraka,\frakb}}$ induced by the isomorphism $\Lie(Tate_{\fraka,\frakb}(q)_{/\Z((\fraka\frakb;\sS))})=\fraka^*\ot_\Z\Lie(\Gm)\iso\fraka^*\ot\Z((\fraka\frakb;\sS))$.
Since $\fraka$ is prime to $p$, the natural inclusion $\fraka^*\ot_\Z\bbmu_{p^n}\hookto\fraka^*\ot_\Z\Gm$ induces a canonical level $p^n$-structure $\eta_{p,\can}\colon \OF^*\ot_\Z\bbmu_{p^n}=\fraka^*\ot_\Z\bbmu_{p^n}\hookto Tate_{\fraka,\frakb}(q)$. Let $\sL_{\fraka,\frakb}=\sL\cdot\MX{\frakb}{}{}{\fraka^{-1}}=\frakb e_1\oplus\fraka^*e_2$. Then we have a level $N$-structure  $\eta_\can^\setp:N^{-1}\sL_\ab/\sL_\ab\isoto Tate_\ab(q)[\frakN]$
over $\Z[\zeta_N]((N^{-1}\fraka\frakb;\sS))$ induced by the fixed primitive $N$-th root of unity $\zeta_N$.
We write $\ul{Tate}_\ab$ for the Tate $\Z((\fraka\frakb;\sS))$-quintuple $(Tate_\ab(q),\lam_\can,\iota_\can,\ol{\eta}^\setp_\can,\eta_{p,\can})$ at $(\fraka,\frakb)$.

\subsubsection{Geometric modular forms}\label{S:GME}We collect here definitions and basic facts of geometric modular forms. The whole theory can found in \cite{Katz:p_adic_L-function_CM_fields} and \cite{Hida:p-adic-automorphic-forms}.
Let $T=\Res_{\OF/\Z}\Gm$ and $\kappa\in\Hom(T,\Gm)$. Let $B$ be an $\bR$-algebra.
Consider $[(\ulA,\lp)]=[(A,\lam,\iota,\lpp,\lp)]\in\Igusa(\plideal)(C)$ (resp. $[(\ulA,\lp)]=[(A,\ollam,\iota,\lpp,\lp)]\in \Igusa(C)$) for a $B$-algebra $C$ with a differential form $\Om$ generating
$H^0(A,\Omega_{A/C})$ over $\OF\ot_\Z C$. A geometric modular form $f$ over $B$ of weight $\kappa$ on $\Igusa(\plideal)$ (resp. $\Igusa$) is a functorial rule of
assigning a value $f(\ulA,\lp,\Om)\in C$ satisfying the following axioms.
\begin{mylist}
\item[(G1)] $f(\ulA,\lp,\Om)=f(\ulA',\lp',\Om')\in C$ if $(\ulA,\lp,\Om)\iso (\ulA',\lp',\Om')$ over $C$,
\item[(G2)]For a $B$-algebra homomorphism $\varphi:C\to C'$, we have
\[f((\ulA,\lp,\Om)\ot_C C')=\varphi(f(\ulA,\lp,\Om)),\]
\item[(G3)]$f((\ulA,\lp,a\Om)=\kappa(a^{-1})f(\ulA,\lp,\Om)$ for all $a\in T(C)=(\OF\ot_\Z C)^\x$,
\item[(G4)]$f(\ul{Tate}_\ab,\Om_\can)\in B\powerseries{(N^{-1}\fraka\frakb)_{\geq 0}}\text{ at all cusps }(\fraka,\frakb)$ in $\Igusa(\plideal)$ (resp. $\Igusa$).
\end{mylist}
For a positive integer $k$, we regard $k\in\Hom(T,\Gm)$ as the character $x\mapsto \bfN_{\cF/\Q}(x)^k,\,x\in O^\x$.
We denote by $\cM_k(\plideal,\lsgN,B)$ (resp. $\cM_k(\lsgN,B)$) the space of geometric modular forms over $B$ of weight $k$ on $\Igusa(\plideal)$ (resp. $\Igusa$).
For $f\in\cM_k(\lsgN,B)$, we write $f|_\plideal\in\cM_k(\plideal,\lsgN,B)$ for $f|_{\Igusa(\plideal)}$.

For each $f\in \cM_k(\lsgN,\C)$, we regard $f$ as a holomorphic Hilbert modular form of weight $k$ and level $\lsgN$ by
\[f(\tau,g_f)=f(\EucA_g(\tau),\ol{\lam_\can},\iota_\C,\ol{\eta}_g,2\pii dz),\]
where $dz$ is the differential form in \eqref{E:7.N}. By GAGA principle, this gives rise to an isomorphism $\cM_{k}(\lsgN,\C)\isoto\bfM_k(\lsgN,\C)$ and  $\cM_{k}(\plideal,\lsgN,\C)\isoto\bfM_k(\plideal,\lsgN,\C)$.
As discussed in \cite[\S 1.7]{Katz:p_adic_L-function_CM_fields}, the evaluation $\f(\ul{Tate}_\ab,\Om_\can)$ is independent of the auxiliary choice of $\sS$ in the construction of the Tate object. Moreover, we have the following important identity which bridges holomorphic modular forms and geometric modular forms
\[\f|_{(\ab)}(q)=\f(\ul{Tate}_\ab,\Om_\can)\in\C\powerseries{(N^{-1}\fraka\frakb)_{\geq 0}}.\]
By the $q$-expansion principle (See \cite{KaiWenLan:FourierExpansion}), if $B$ is $\baseR$-algebra in $\C$ and $\f \in\bfM_k(\plideal,\lsgN,B)\iso \cM_k(\plideal,\lsgN,B)$, then $\f|_\plideal\in\cM_{k}(\plideal,\lsgN,B)$.

\subsubsection{\padic modular forms}Let $B$ be a \padic $\baseR$-algebra in $\Cp$. Let $V(\plideal,\opcpt,B)$ be the space of Katz \padic modular forms over $B$ defined by
\[V(\plideal,\opcpt,B):=\prolim_m\dirlim_n H^0(\Igusa(\plideal){}_{/B/p^mB},\cO_{\Igusa}).\]
 In other words, Katz \padic modular forms are formal functions on the Igusa tower. Let $C$ be a $B/p^mB$-algebra. For each $C$-point $[(\ulA,\lp)]=[(A,\lam,\iota,\lpp,\lp]\in\Ig_\opcpt(\plideal)(C)=\prolim_n\Igusa(\plideal)(C)$, the $p^\infty$-level structure $\lp$ induces an isomorphism $\lp_*:\OF^*\ot_\Z C\iso \Lie A$, which in turns gives rise to a generator $\Om(\lp)$ of $H^0(A,\Omega_A)$ as a $\OF\ot_\Z C$-module.
We thus have a natural injection\beq\label{E:padicavatar.N}
\begin{aligned}\cM_k(\plideal,\lsgN,B)&\hookto V(\plideal,\opcpt,B)\\
f&\mapsto \wh{f}(\ulA,\lp):=f(\ulA,\lp,\Om(\lp))
\end{aligned}\eeq
which preserves the $q$-expansions in the sense that $\wh f|_{(\ab)}(q):=\wh f(\ul{Tate}_\ab)=f|_{(\ab)}(q)$. We will call $\wh{f}$ the \padic avatar of $f$.

\subsection{Hecke action.}Let $h\in G(\AFf^\setpN)$. Put ${}_hK=hK h^{-1}$. We define $\smid h:\cI^\setp_{{}_hK,n}\isoto\cI^\setp_{K,n}$ by
\[(A,\ollam,\iota,\ol{\eta}^\setp,\lp)\mapsto \ulA\smid h=(A,\ollam,\iota,\ol{\eta}^\setp h,\lp).\]
Then $\smid h$ induces an $\baseR$-isomorphism $\Igusa\isoto
\Ig_{{}_h\opcpt,n}$. In addition, $\smid h$ induces an $\baseR$-isomorphism $\Igusa(\plideal)\isoto\Ig_{{}_hK,n}(\plideal(h))$ with $\plideal(h)=\plideal\det(h)^{-1}$ and hence
$\cM_k(\plideal(h),\lsgN,B)\isoto\cM_k(\plideal,{}_h\lsgN,B)$ for every $\baseR$-algebra $B$.

Using the description of the complex points $\sh^\setp_{K}(\C)$ in \secref{S:cpx}, the two pairs $(\ul{\EucA_g(\tau)}\smid h,\Om)$ and $(\ul{\EucA_{gh}(\tau)},\Om)$ are $\Z_\setp$-isogenous
, so we have the isomorphism:
\beq\label{E:HeckeAct}\begin{aligned}
\bfM_k(\plideal(h),\lsgN,\C)&\isoto \bfM_k(\plideal,{}_h\lsgN,\C)\\
\f&\mapsto\f|h(\tau,g)=\f(\tau,gh).\end{aligned}\eeq


\section{CM points}\label{S:CMpoint} \subsection{}\label{S:CM1} In this section, we give an adelic description of CM points in Hilbert modular varieties.
Fix a prime-to-$p$ integral ideal $\frakC$ of $\OK$. Decompose $\frakC=\frakC^+\frakC^-$, where $\frakC^+=\Csplit\Csplit_c$ is a product of split primes in
$\cK/\cF$ such that $(\Csplit,\Csplit_c)=1$ and $\Csplit\subset\Csplit_c^c$, and $\frakC^-$ is a product of non-split primes in $\cK/\cF$. Let \[\frakD:=p\frakC\frakC^cD_{\cK/\cF}.\] We choose $\skewhf\in\cK$ such that
\begin{itemize}
\item[(d1)] $\skewhf^c=-\skewhf$ and
$\Im\sg(\skewhf)>0$ for all $\sg\in\Sg$,
\item[(d2)] $\frakc(\OK):=\cD_{\cF}^{-1}(2\skewhf\cD_{\cK/\cF}^{-1})$ is prime to $\frakD$.
\end{itemize}
Let
$\skewhf^\Sg:=(\sg(\skewhf))_{\sg\in\Sg}\in X^+$. Let $D=-\skewhf^2\in \cF_+$ and define
$\rho:\cK\hookto M_2(\cF)$ by
\[\rho(a\skewhf+b)=\MX{b}{-D a}{a}{b}.\]Consider the isomorphism $q_\skewhf:\cK\isoto \cF^2=V$ defined by $q_\skewhf(a\skewhf+b)=ae_1+be_2$.
Note that $(0,1)\rho(\al)=q_\skewhf(\al)$
and $q_\skewhf(x \al)=q_\skewhf(x)\rho(\al)$ for $\al,x\in \cK$. Let $\C(\Sg)$ be the $\cK$-module whose underlying space is
$\C^\Sg$ with the $\cK$-action given $\al(x_\sg)=(\sg(\al)x_\sg)$.
Then we have a canonical isomorphism $\cK\ot_\Q\R=\C(\Sg)$ and an isomorphism $p_\skewhf=q_\skewhf^{-1}:V\ot_\Q\R\isoto \cK\ot_\Q\R=\C(\Sg)$.

\subsection{}\label{S:different}
For each split place $v|p\Csplit\Csplit^c$, we decompose $v=w\wbar$ into two places $w$ and $\wbar$ of $\cK$ with $w|\Csplit\Sg_p$. Here $w|\Csplit\Sg_p$ means $w|\Csplit$ or $w\in\Sg_p$. Let $e_{w}$ (resp.
$e_{\wbar}$) be the idempotent associated to $w$ (resp. $\wbar$). Then $\stt{e_w,e_{\wbar}}$ gives an $\OFv$-basis of $\OKv$. Let $\skewhf_w\in \cF_v$ such that $\skewhf=-\skewhf_we_{\wbar}+\skewhf_w e_w$.

For a non-split place $v$ and $w$ the place of $\cK$ above $v$, we fix a $\OFv$-basis $\stt{1,\bftheta_v}$ such that $\bftheta_v$ is a uniformizer if $v$ is ramified and $\ol{\bftheta_v}=-\bftheta_v$ if $v\ndivides 2$. We let $t_v=\bftheta_v+\ol{\bftheta_v}$ and let $\delta_v:=\bftheta_v-\ol{\bftheta_v}$ be a generator of the relative different $\cD_{\cK_w/\cF_v}$.

Fix a finite idele $d_\cF=(d_{\cF_v})\in \adelef$ such that $\il_\cF(d_\cF)=\cD_\cF$. By condition (d2), we may choose $d_{\cF_v}=2\skewhf\delta_v^{-1}$ if $v|D_{\cK/\cF}\frakC^-$ (resp. $d_{\cF_v}=-2\skewhf_w$ if $w|\Csplit\Sg_p$).
\subsection{A good level structure}
We shall fix a choice of a basis $\stt{e_{1,v},e_{2,v}}$ of $R\ot_\OF \OF_v$ for each finite
place $v$ of $\cF$.
If $v\ndivides p\frakC\frakC^c$, we choose
$\stt{e_{1,v},e_{2,v}}$ in $\OK\ot\OF_v$ such that
$R\ot_{\OF}\OF_v=\OF_v e_{1,v}\oplus \OF_v^* e_{2,v}$. Note that $\stt{e_{1,v},e_{2,v}}$ can be taken to be $\stt{\skewhf,1}$
except for finitely many $v$. If $v|p\Csplit\Csplit^c$, let $\stt{e_{1,v},e_{2,v}}=\stt{e_{\wbar},d_{\cF_v}\cdot e_{w}}$ with $w|\Csplit\Sg_p$.
If $v$ is inert or ramified, let
$\stt{e_{1,v},e_{2,v}}=\stt{\bftheta_v,d_{\cF_v}\cdot 1}$.
For $v\in\bdh$, we let $\cmptv$ be
the element in $\GL_2(\cF_v)$ such that
$e_i\cmptv'=q_\skewhf(e_{i,v})$. For
$v=\sg\in \bda$, let $\cmptv=\MX{\Im\sg(\skewhf)}{0}{0}{1}$. We define $\cmpt=\prod_{v}\cmptv\in\GL_2(\A_\cF)$. Let $\cmpt_f$ be the finite component of $\cmpt$. By the definition of
$\cmpt$, we have
\[\cmpt_f*(\sL\ot_\Z\Zhat)=(\sL\ot_\Z\Zhat)\cdot\cmpt_f'=q_\skewhf(\OK\ot_\Z\Zhat).\]
The matrix representation of $\cmptv$ according to $\stt{e_1,e_2}$ for $v|\frakD$ is given as follows:
\beq\begin{aligned}\label{E:cm.N}
\cmptv&=\MX{d_{\cF_v}}{-2^{-1}t_v}{0}{d_{\cF_v}^{-1}}\text{ if $v|D_{\cK/\cF}\frakC^-$},\\
\cmptv&= \MX{\frac{d_{\cF_v}}{2}}{-\onehalf}{\frac{-d_{\cF_v}}{2\skewhf_w}}{\frac{-1}{2\skewhf_w}}=\MX{-\skewhf_w}{-\onehalf}{1}{\frac{-1}{2\skewhf_w}}
\text{ if }v|p\Csplit\Csplit^c\text{ and }w|\Csplit\Sg_p.\\
\end{aligned}\eeq

\subsection{}\label{SS:CMpoint}
\def\CMring{{\wh{W}}}
The alternating pairing $\pairing:\cK\x\cK:\to\cF$
defined by $\pair{x}{y}=(c(x)y-xc(y))/2\skewhf$ induces an isomorphism
$\OK\wedge_{\OF} \OK=\frakc(\OK)^{-1}\cD_\cF^{-1}$ for
ideal $\frakc(\OK)=\cD_\cF^{-1}(2\skewhf \cD_{\cK/\cF}^{-1})$. On the other hand, by the equation
\[\cD_\cF^{-1}\det(\cmpt_f)=\wedge^2\sL \cmpt_f'=\wedge^2\OK=\frakc(\OK)^{-1}\cD_\cF^{-1},\]
we also have $\frakc(\OK)=(\det(\cmpt_f))^{-1}$.

For $a\in (\AKf^\setpN)^\x$, put $\fraka=\il_\cK(a)$ and $\frakc(a):=\frakc(\OK)\rmN_{\cK/\cF}(\fraka)^{-1}$. We let
\[(\ul{A}(a),\lp(a))_{/\C}=(\EucA_{\rho(a)\cmpt_f}(\skewhf^\Sg),\pairing_\can,\iota_\can,\lpp(a),\lp(a))\] be the $\frakc(a)$-polarized $\C$-quintuple associated to $(\skewhf^\Sg,\rho(a)\cmpt_f)$ as in \secref{S:cpx}.
Then
$(\ul{A}(a),\lp(a))_{/\C}$ is an
abelian variety with CM by the field $\cK$ and gives rise to a complex point $[(\skewhf^\Sg,\rho(a)\cmpt_f)]$ in $\Ig_\opcpt(\plideal(a))(\C)$. Let $W$ be the maximal unramified extension of $\Zp$ in $\Cp$ and let $\CMring$ be the \padic completion of $W$. By the theory of CM abelian varieties, the $\C$-quadruple $\ul{A}(a)_{/\C}$ is rational over a number field $L$ (See \cite[18.6,\,21.1]{Shimura:ABV-with-CM}), which in turn descends to a $W$-quadruple $\ul{A}(a)$ by a theorem of Serre-Tate. In addition, since the CM type $\Sg$ is $p$-ordinary, $\ulA(a)\ot\Fpbar$ is an ordinary abelian variety (\cf\cite[5.1.27]{Katz:p_adic_L-function_CM_fields}), and hence $\lp(a)$ descends to a level $p^\infty$-structure over $\CMring$.
We obtain $x(a)\in\Ig_{\opcpt}(\frakc(a))(\CMring)\hookto\Ig_{\opcpt}(\CMring)$. This collection of points $x(a)$ with $a\in (\AKf^\setpN)^\x$ is called \emph{CM points} in Hilbert modular varieties.

\section{Katz Eisenstein measure}\label{S:ES}\subsection{}\label{S:41}
In this section, we recall the construction of \padic $L$-functions for CM fields following \cite{Katz:p_adic_L-function_CM_fields} and \cite{HidaTilouine:KatzPadicL_ASENS}.
First we give the construction of a \padic Eisenstein measure of Katz, Hida and Tilouine from representation theoretic point of view.
This construction is inspired by \cite{HLS}.

Let $\ads$
be a Hecke character of $\cK^\x$ with infinity type
$k\Sg+\kappa(1-c)$, where $k\geq 1$ is an integer and $\kappa=\sum \kappa_\sg\sg\in \Z[\Sg]$, $\kappa_\sg\geq 0$. We suppose that $\frakC$ is divisible by the prime-to-$p$ conductor of $\ads$. Put \[\nads=\ads\Abs^{-\onehalf}_{\AK}\text{ and }\ads_+=\ads|_{\AF^\x}.\]
Let $\opcpt^0_\infty:=\prod_{v\in\bda}\SO(2,\R)$ be a maximal compact subgroup of $G(\cF\ot_\Q\R)$. For $s\in\C$, we let $I(s,\lam_+)$ denote the space consisting of smooth and $\opcpt^0_\infty$-finite functions $\phi:G(\AF)\to \C$ such that \[\phi(\MX{a}{b}{0}{d}g)=\ads_+^{-1}(d)\abs{\frac{a}{d}}_{\AF}^s\phi(g).\]
Conventionally, functions in $I(s,\ads_+)$ are called \emph{sections}. Let $B$ be the upper triangular subgroup of $G$. The adelic Eisenstein series associated
to a section $\phi\in I(s,\ads_+)$ is defined by
\[E_\A(g,\phi)=\sum_{\gamma\in B(\cF)\bksl G(\cF)}\phi(\gamma g).\]
The series $E_\A(g,\phi)$ is absolutely convergent for $\Re s\gg 0$.
\subsection{Fourier coefficients of Eisenstein series}We put
$\bdw=\MX{0}{-1}{1}{0}$. Let $v$ be a place of $\cF$ and let $I_v(s,\ads_+)$ be the local constitute of $I(s,\ads_+)$ at $v$. For $\phi_v\in I_v(s,\ads_+)$ and $\beta\in \cF_v$, we recall that the $\beta$-th local Whittaker integral $W_\beta(\phi_v,g_v)$ is defined by \begin{align*}W_\beta(\phi_v,g_v)=&
\int_{\cF_v}\phi_v(\bdw\MX{1}{x_v}{0}{1}g_v)\addchar(-\beta x_v)dx_v,
\intertext{ and the intertwining operator $M_\bdw$ is defined by}
M_\bdw\phi_v(g_v)=&\int_{\cF_v}\phi_v(\bdw \MX{1}{x_v}{0}{1}g_v)dx_v.\end{align*}
By definition, $M_\bdw\phi_v(g_v)$ is the $0$-th local Whittaker integral. It is well known that local Whittaker integrals converge absolutely for $\Re s\gg 0$, and have meromorphic continuation to all $s\in\C$.

If $\phi=\ot_v\phi_v$ is a decomposable section, then it is
well known that $E_\A(g,\phi)$ has the following Fourier expansion:
\beq\label{E:WF.N}\begin{aligned}&E_\A(g,\phi)=\phi(g)+M_\bdw\phi(g)+\sum_{\beta\in\cF}W_\beta(E_\A,g),\text{ where }\\
M_\bdw\phi(g)=&\frac{1}{\sqrt{\abs{D_\cF}_\R}}\cdot \prod_vM_\bdw\phi_v(g_v)\,;\,
W_\beta(E_\A,g)=\frac{1}{\sqrt{\abs{D_\cF}_\R}}\cdot\prod_vW_\beta(\phi_v,g_v).\end{aligned}\eeq
The sum $\phi(g)+M_\bdw\phi(g)$ is called the \emph{constant term} of $E_\A(g,\phi)$. The general analytic properties of the local Whittaker integrals and the constant term can be found in \cite[\S\,3.7]{Bump:AutoRep}.

\subsection{Choice of the local sections}\label{S:ES_notation}
In this subsection, we recall the choice of sections made in \cite[\S 4.3]{Hsieh:Hecke_CM}. We first introduce some notation. Let $v$ be a place of $\cF$.
Let $L/\cF_v$ be a finite extension and let $d_L$ be a generator of the absolute different of $L$. Let $\addchar_L:=\addchar\circ\Tr_{L/\cF_v}$. If $\mu:L^\x\to\C^\x$ is a character, define
\[a(\mu)=\inf\stt{n\in \Z_{\geq 0}\mid
\mu(x)=1 \text{ for all }x\in {(1+\uf_L^n \cO_L)\cap \cO_L^\x}}.\]
We recall that the epsilon factor $\ep(s,\mu,\addchar_L)$ in \cite{Tate:Number_theoretic_bk} is defined by
\begin{align*}\ep(s,\mu,\addchar_L)&=\abs{c}_L^s\int_{c^{-1}\cO_L^\x}\mu^{-1}(x)\addchar_L(x)d_Lx,\,c=d_L\uf_L^{a(\mu)}.\end{align*}
Here $d_Lx$ is the Haar measure on $L$ self-dual with respect to $\addchar_L$.
The local root number $W(\mu)$ is defined by \[W(\mu):=\ep(\onehalf,\mu,\addchar_L)\] (\cf \cite[p.281 (3.8)]{Murase-Sugano:Local_theory_primitive_theta}). It is well known that $\abs{W(\mu)}_\C=1$ if $\mu$ is unitary. If $\vp$ is a \BS function on $L$, the zeta integral $Z(s,\mu,\vp)$ is given by
\[Z(s,\mu,\vp)=\int_{L}\vp(x)\mu(x)\abs{x}_L^s\dx x\quad(s\in\C).\]

To simplify our notation, we put $\Fv=\cF_v$ (resp.
$\Kv=\cK\ot_{\cF}\cF_v$) and let $d_F=d_{\cF_v}$ be the fixed generator of the different $\cD_\cF$ of $\cF/\Q$ in \subsecref{S:different}. Write $\ads$ (resp. $\ads_+$, $\nads$) for $\ads_v$ (resp. $\ads_{+,v}$, $\nads_v$). If $v\in\bdh$, we let $\OFv=\cO_{\Fv}$ (resp.
$\OKv=\OK\ot_{\OF}\OFv$) and let $\uf=\uf_v$. For a set $Y$, denote by $\bbI_Y$ the characteristic function of $Y$.
\\

\emph{The archimedean case}: Let $v=\sg\in \Sg$ and $\Fv=\R$. For $g\in \GL_2(\R)$, we put
\begin{align*}\bfdelta(g)=\abs{\det (g)}\cdot\abs{J(g,i)\ol{J(g,i)}}^{-1}.\end{align*}
Define the sections $\phi^h_{k,s,\sg}$ of weight $k$ and $\phi^{\nh}_{k,\kappa_\sg,s,\sg}$ of weight $k+2\kappa_\sg$ in $I_v(s,\ads_+)$ by
\begin{align*}\phi^h_{k,s,\sg}(g)=&J(g,i)^{-k}\bfdelta(g)^s,\\
\phi^{\nh}_{k,\kappa_\sg,s,\sg}(g)=&J(g,i)^{-k-\kappa_\sg}\ol{J(g,i)}^{\kappa_\sg}\bfdelta(g)^s.
\end{align*}
The intertwining operator $M_\bdw\phi_{k,s,\sg}$ is given by
\beq\label{E:M2.H}M_\bdw\phi^h_{k,s,\sg}(g)=i^{k}(2\pi)\frac{\Gamma(k+2s-1)}{\Gamma(k+s)\Gamma(s)}\cdot \ol{J(g,i)}^k\det(g)^{-k}\bfdelta(g)^{1-s}.\eeq
\emph{The case $v\ndivide \frakD$ or $v|p\Csplit\Csplit^c$}: Denote by $\cS(F)$ and (resp. $\cS(F\oplus F)$) the space of \BS functions on $F$ (resp. $F\oplus F$).
Recall that the Fourier transform $\wh\vp$ for $\vp\in\cS(F)$ is defined by
\[\wh\vp(y)=\int_{\Fv}\vp(x)\addchar(yx)dx.\]
For a character $\mu: F^\x\to\C^\x$, we define the function $\vp_\mu\in\cS(F)$ by
\[\vp_\mu(x)=\ch_{\OFv^\x}(x)\mu(x).\]
If $v|p\Csplit\Csplit^c$ is split in $\cK$, write $v=w\wbar$ with $w|\Csplit\Sg_p$, and set
\[\vp_{w}=\vp_{\ads_{w}}\text{ and }\vp_{\wbar}=\vp_{\ads_{\wbar}^{-1}}.\]
To a \BS function $\Phi\in\cS(F\oplus F)$, we can associate a Godement section $f_{\Phi,s}\in I_v(s,\ads_+)$ defined by
\beq\label{E:Godement.V}f_{\Phi,s}(g):=\abs{\det
g}^s\int_{F^\x}\Phi((0,x)g)\ads_+(x)\abs{x}^{2s}\dx x.\eeq
Define the Godement section $\Section$ by
\beq\label{E:sectionsplit.V}\Section=f_{\Phi^0_v,s},\text{ where }\Phi^0_v(x,y)=\begin{cases}
\bbI_{\OFv}(x)\bbI_{\OFv^*}(y)&\cdots v\ndivides \frakD,\\
\vp_{\wbar}(x)\wh\vp_{w}(y)&\cdots v\mid p\Csplit\Csplit^c.
\end{cases}\eeq
We remark that the choice of $\Phi^0_v$ has its origin in \cite[5.2.17]{Katz:p_adic_L-function_CM_fields} (\cf \cite[p.209]{HidaTilouine:KatzPadicL_ASENS} and \cite[3.3.4]{HLS}).
For every $u\in \OFv^\x$ with $v|p$, let $\vp_{\wbar}^1$ and $\vp^\class{u}_w\in\cS(F)$ be \BS functions given by
\[\vp_{\wbar}^1(x)=\bbI_{1+\uf \OFv}(x)\ads_{\wbar}^{-1}(x)\text{ and }\vp^\class{u}_w(x)=\bbI_{u(1+\uf \OFv)}(x)\ads_w(x).\]
Define $\Phi^\class{u}_v\in\cS(F\oplus F)$ by
\beq\label{E:sectiontorsion.V}\Phi^\class{u}_v(x,y)=\frac{1}{\vol(1+\uf\OFv,\dx x)}\vp^1_{\wbar}(x)\wh\vp^\class{u}_{w}(y)=(\abs{\uf}^{-1}-1)\vp^1_{\wbar}(x)\wh\vp^\class{u}_{w}(y).\eeq

\emph{The case $v|D_{\cK/\cF}\frakC^-$}:
In this case, $\Kv$ is a field and $G(F)=B(F)\rho(\Kv^\x)$. Let $\Section$ be the unique smooth section in $I_v(s,\ads_+)$ such that
\beq\label{E:Dinert.N}\Section(\MX{a}{b}{0}{d}\rho(z)\cmptv)=L(s,\ads_v)\cdot \ads_+^{-1}(d)\abs{\frac{a}{d}}^s\cdot\ads^{-1}(z)\quad(b\in
B(F),\,z\in \Kv^\x),\eeq
where $L(s,\ads_v)$ is the local Euler factor of $\ads_v$, and $\cmptv$ is defined as in \eqref{E:cm.N}. Note that $L(s,\ads_v)=1$ unless $v\ndivides\frakC$ is ramified in $\cK$.

\subsection{The local Whittaker integrals}\label{S:Local_ES}
\def\wbar{\ol{w}}
We summarize the formulae of the local Whittaker integrals of the special local sections $\Section$ in the following proposition.
\begin{prop}\label{P:Fourier.M}The local Whittaker integrals of $\Section$ are given as follows:
\begin{align*}
\intertext{ If $\sg\in\bfa$, then }
W_\beta(\phi^h_{k,s,\sg},\MX{y}{x}{0}{1})|_{s=0}&=\frac{(2\pii)^k}{\Gamma(k)}\sg(\beta)^{k-1}\exp(2\pi
i\sg(\beta)(x+iy))\cdot\ch_{\R_+}(\sg(\beta)).\\
\intertext{ If $v\in\bdh$ and $v\ndivide \frakD$, then }
W_\beta(\Section,\MX{1}{}{}{\bfc_v^{-1}})|_{s=0}&=\sum_{i=0}^{v(\beta\bfc_v)}\ads_+(\uf^i\bfc_v)\abs{\uf}^{-i}\cdot \abs{\cD_{\cF}}^{-1}\bbI_{\OFv}(\beta\bfc_v).
\intertext{If $v|\frakD$, then }W_\beta(\Section,1)|_{s=0}&=\begin{cases}\ads_w(\beta)\bbI_{\OFv^\x}(\beta)\cdot\abs{\cD_{\cF}}^{-1}&\cdots v|p\Csplit\Csplit^c,\,w|\Csplit\Sg_p,\\
 L(0,\ads_v)\cdot A_\beta(\ads_v)\cdot \abs{\cD_\cF}^{-1}\addchar(-2^{-1}t_vd_{\cF_v}^{-1})&\cdots v|\frakC^-D_{\cK/\cF},
\end{cases}
\end{align*}
where
\beq\label{E:Abeta.V}A_\beta(\ads_v)=\int_{\cF_v}\ads_v^{-1}(x+2^{-1}\UF)\addchar(-d_{\cF_v}^{-1}\beta x)dx.\eeq
If $v=w\wbar$ with $w\in\Sg_p$, then we have
\begin{align*}W_\beta(f_{\Phi^\class{u}_v},1)|_{s=0}=&\ads_w(\beta)\bbI_{u(1+\uf_v\OFv)}(\beta)\cdot \abs{\cD_\cF}^{-1}\quad(u\in \OFv^\x).
\intertext{In particular, we have }W_\beta(\Section,1)|_{s=0}=&\sum_{u\in\cU_v}W_\beta(f_{\Phi^\class{u}_v},1)|_{s=0},\end{align*}
where $\cU_v$ is the torsion subgroup of $\OFv^\x$.
\end{prop}
\begin{proof}
The formulas of the local integrals of $\Section$ can be found in \cite[\S 4.3]{Hsieh:Hecke_CM}, and the computation of the local integral $W_\beta(f_{\Phi^\class{u}_v},1)$ is straightforward. We omit the details.
\end{proof}

\begin{Remark}\label{R:2.V}We remark that the local Whittaker integrals at all finite places belong to a finite extension $\bR$ of $\cO^\x_{\cF,\setp}$. Indeed, it is well known that $\ads|_{\AKf^\x}$ takes value in a number field $L$ and $\ads_v$ takes value in $\cO_{L,\setp}^\x$ for each finite $v\ndivide p$, so the local Whittaker integrals $W_\beta(\Section,\MX{1}{}{}{\bfc_v^{-1}})|_{s=0}$ belongs to $\cO_{L,\setp}$ whenever $v\ndivide \frakC^-D_{\cK/\cF}$.
Suppose that $v|\frakC^-D_{\cK/\cF}$.
Then there is an $M_{\UF}\geq 0$ such that $\abs{x+2^{-1}\UF}_E< \abs{\uf}^{M_{\UF}}$ for all $x\in \uf^{-M}\OFv$ as $\UF\not\in \cF_v$, and we have
\[A_\beta(\ads_v)=\abs{\uf_v}^{M_\UF+M}\sum_{x\in \OFv/(\uf^{M_\UF+2M})}\ads_v^{-1}(x\uf^{-M}+2^{-1}\UF)\addchar^\circ(\beta x\uf_v^{-M})\]
for $M\geq \max\stt{v(\frakC),v(\frakC)-v(\beta)}$. We may thus enlarge $L$ such that $A_\beta(\ads_v)$ takes value in $\cO_{L,\setp}$ as well. We shall fix this $L$ and let $\bR:=\cO_{L,\setp}$ in the remainder of this paper.
\end{Remark}
\def\nh{{n.h.}}
\subsection{Normalized Eisenstein series}\label{S:normalization}We introduce some normalized Eisenstein series. \begin{defn} For each \BS function $\Phi=\ot_{v|p}\Phi_v$ on $\cF_p\oplus\cF_p$, we define
\[\phi^{\bullet}_{\ads,s}(\Phi)=\bigot_{\sg\in
\bda}\phi_{k,s,\sg}^\bullet\bigot_{\substack{v\in\bdh,\\v\ndivide p}}\Section\bigot_{v|p}f_{\Phi_v,s},\,\bullet=h,\,\nh \]
and define the adelic Eisenstein series $E^\bullet_{\ads}(\Phi)$ by
\begin{align*}
E^\bullet_{\ads}(\Phi)(g)&=E_\A(g,\phi_{\ads,s}^{\bullet}(\Phi))|_{s=0},\,\bullet=h,\nh.
\end{align*}
We define the holomorphic (resp. nearly holomorphic) Eisenstein series $\holES(\Phi)$ (resp. $\bbE^\nh_\ads(\Phi)$) by
\beq\label{E:ES1.N}\begin{aligned}
\bbE^h_{\ads}(\Phi)(\tau,g_f)&:=\frac{\Gamma_\Sg(k\Sg)}{\sqrt{\abs{D_\cF}_\R}(2\pi
i)^{k\Sg}}\cdot E^h_{\ads}(\Phi)\left(g_\infty,
g_f\right)\cdot \ul{J}(g_\infty,\bfi)^{k\Sg},\\
\bbE^\nh_{\ads}(\Phi)(\tau,g_f)&:=\frac{\Gamma_\Sg(k\Sg)}{\sqrt{\abs{D_\cF}_\R}(2\pi
i)^{k\Sg}}\cdot E^\nh_{\ads}(\Phi)\left(g_\infty,
g_f\right)\cdot \ul{J}(g_\infty,\bfi)^{k\Sg+2\kappa}(\det g_\infty)^{-\kappa},\\
 &\quad ((\tau,g_f)\in X^+\x G(\AFf),\,g_\infty\in G(\cF\ot_\Q\R),\,g_\infty\bfi=\tau,\,\bfi=(i)_{\sg\in\Sg}).
\end{aligned}\eeq
\end{defn}
Let $\Phi^0_p=\ot_{v|p}\Phi^0_v$ be the \BS function on $\cF_p\oplus\cF_p$ defined in \eqref{E:sectionsplit.V}. Set
\[\holES=\holES(\Phi^0_p)\text{ and }\bbE^\nh_{\ads}=\bbE^\nh_{\ads}(\Phi^0_p).\]
For every $u=(u_v)_{v|p}\in \prod_{v|p}\OFv^\x=\OF_p^\x$, let $\Phi^\class{u}_p=\ot_{v|p}\Phi^\class{u_v}_v$ be the \BS function defined in \eqref{E:sectiontorsion.V}
and set
\[\Eadsu=\holES(\Phi^\class{u}_p).\]
We choose $\frakN=\bfN_{\cK/\Q}(\frakC\cD_{\cK/\cF})^m$ for a sufficiently large integer $m$ so that $\Section$ are invariant by $U(\frakN)$ for every $v|\frakN$, and put $\opcpt:=U(\frakN)$. Then the section $\phi_{\ads,s}(\Phi^\class{u}_p)$ is invariant by $\lsgN$ for a sufficiently large $n$.

Let $\bfc=(\bfc_v)\in\AF^\x$ such that $\bdc_v=1$ at $v|\frakD$ and let $\frakc=\il_\cF(\bdc)$. For each $\beta\in\cF_+$, we define the \emph{prime-to-$p$ $\beta$-th Fourier coefficient} $\bfa_\beta^\setp(\lam,\frakc)$ by
\beq\begin{aligned}\label{E:FA.V}\bfa_\beta^\setp(\ads,\frakc):=&\frac{1}{\abs{D_\cF}_\R\abs{D_\cF}_{\Qp}}\cdot \bfN_{\cF/\Q}(\beta^{-1})\cdot \prod_{v\ndivide p}W_\beta(\Section,\MX{1}{}{}{\bfc_v^{-1}})|_{s=0}\cdot \bbI_{O^\x_p}(\beta)\\
=&\beta^{(k-1)\Sg}\prod_{w|\Csplit}\lam_{w}(\beta)\bbI_{\OFv^\x}(\beta)\cdot\prod_{v\ndivide \frakD}\left(\sum_{i=0}^{v(\bfc_v\beta)}\ads_{+,v}\Abs^{-1}(\uf_v^i)\right)\\
 &\times\prod_{v|\frakC^- D_{\cK/\cF}}L(0,\ads_v)A_\beta(\ads_v)\addchar_v(-2^{-1}t_vd_{\cF_v}^{-1}).
\end{aligned}\eeq
The last equality follows from the formulae of the local Whittaker integrals in \propref{P:Fourier.M}. It is clear that $\bfa_\beta^\setp(\ads,\frakc)$ belongs to $\bR$.
\begin{prop}\label{P:1.N}The Eisenstein series $\Eadsu$ belongs to $\bfM_k(\lsgN,\C)$. The $q$-expansion of $\Eadsu$ at the cusp $(O,\frakc^{-1})$ has no constant term and is given by
\[\Eadsu|_{(\OF,\frakc^{-1})}(q)=\sum_{\beta\in (N^{-1}\frakc^{-1})_+}\bfa_\beta(\Eadsu,\frakc)\cdot q^\beta\in \bR\powerseries{(N^{-1}\frakc^{-1})_+},\]
where the $\beta$-th Fourier coefficient $\bfa_\beta(\Eadsu,\frakc)$ is given by
\beq\label{E:fourierformula.E}\begin{aligned}\bfa_\beta(\Eadsu,\frakc)=&\bfa_\beta^\setp(\ads,\frakc)\cdot \beta^{k\Sg}\ads_{\Sg_p}(\beta)\bbI_{u(1+\uf_p\OFp)}(\beta)\\
&(\lam_{\Sg_p}(\beta)=\prod_{w\in\Sg_p}\lam_w(\beta),\,\uf_p=(\uf_v)_{v|p}).\end{aligned}\eeq
Therefore, $\Eadsu|_\plideal\in\cM_k(\frakc,\lsgN,\bR)$ and
\[\holES|_\plideal=\sum_{u\in\cU_p}\Eadsu|_\plideal,\]
where $\cU_p$ is the torsion subgroup of $\OFp^\x$.
\end{prop}
\begin{proof}
By the definition of the local sections $\phi_{\ads,s,v}$ for $v|p$ we find that
\[\phi_{\ads,s,v}(1)=0\,;\,M_\bdw\phi_{\ads,s,v}(1)|_{s=0}=0,\]
so $\Eadsu$ has no constant term. Therefore, we can derive the $q$-expansion of $\holES$ from the equations \eqref{E:WF.N}, \eqref{E:FA.V} and \propref{P:Fourier.M}. To verify the second assertion, note that the Fourier coefficients $\bfa_\beta(\holES,\frakc)$ can be written as
\[\bfa_\beta(\holES,\frakc)=\bfa_\beta^\setp(\ads,\frakc)\cdot \beta^{k\Sg}\ads_{\Sg_p}(\beta)\bbI_{\OF_p^\x}(\beta).\]
Thus, we have
\[\bfa_\beta(\holES,\frakc)=\sum_{u\in\cU_p}\bfa_\beta(\Eadsu,\frakc).\]
This completes the proof.
\end{proof}
\begin{Remark}\label{R:1.V}An important feature of our Eisenstein series $\bbE^h_{\ads,u}$ and $\bbE^h_{\ads}$ is that they are \emph{toric} Eisenstein series of eigencharacter $\lam$. In other words, they are eigenforms of the Hecke action $|[a]:=|(\cmpt_f^{-1}\rho(a)\cmpt_f)$ for a class of ideles $a\in \cT:=\prod_{v\in\bdh}'\cT_v\subset\AKf^\x$, where
\[\cT_v=\begin{cases}\cO^\x_{\cK_v}\cF_v^\x&\text{ if $v$ is split,}\\
\cK_v^\x&\text{ if $v$ is non-split.}\end{cases}\] More precisely, from the definitions of the sections $\Section$ and $f_{\Phi_p^\class{u}}$, it is not difficult to deduce that
\beq\label{E:rightinv.V} \begin{aligned}\bbE^h_{\ads}|[a]=&\ads^{-1}(a)\bbE^h_{\ads};\quad \Eadsu|[a]=\ads^{-1}(a)\bbE^h_{\ads,u.a^{1-c}}\quad(a\in\cT),
\end{aligned}\eeq
where $u.a^{1-c}:=ua_{\Sg_p}a_{\Sgbar_p}^{-1}\in\OF_p^\x$. The above equation will play an important role in the proof of \thmref{T:1.V}.
\end{Remark}
\subsection{\padic Eisenstein measure}\label{SS:ESmeasure.V}
For every integral ideal $\fraka$ of $\OK$, we put \[U_\cK(\fraka)=\stt{a\in (\OK\ot_\Z\Zhat)^\x\mid a\con 1\pmod{\fraka}}.\]
Let $\Glv$ be the ray class group of $\cK$ modulo $\frakC p^\infty$. Then the reciprocity law $\rec_\cK$ induces an isomorphism:
\[\rec_\cK:\prolim_n \cK^\x\bksl \AKf^\x/U_\cK(\frakC p^n)\iso \Glv.\]
Let $\cC(\Glv,\Zbarp)$ be the space of continuous $\Zbarp$-valued functions on $\Glv$. Define a subset $\frakX^+$ of locally algebraic \padic characters by
\[\frakX^+=\stt{\hatads:\Glv\to\Zbarp^\x\mid \text{ $\ads$ has infinity type of $k\Sg$, }k\geq 1}.\]
Then $\frakX^+$ is a Zariski dense subset in $\cC(\Glv,\Zbarp)$. Let $\cZ_1$ be the subgroup of $\AKf^\x$ given by
\beq\label{E:A}\cZ_1=\OK_p^\x\x (\AFf^{(\frakD)})^\x \prod_{v|\Csplit\Csplit^c} \OFv^\x\prod_{w|\cD_{\cK/\cF}\frakC^-}\cK^\x_w .\eeq
Let $Z_1:=\rec_\cK(\cZ_1)$ be a subgroup of $Z(\frakC)$.

We write $\holESc$ for $\holES|_\plideal$ and let $\EucE_{\ads,\plideal}:=\wh\holESc$ be the \padic avatar of $\holESc$. Let $\stt{\Katzd(\sg)}_{\sg\in\Sg}$ be the Dwork-Katz \padic differential operators on \padic modular forms (\cite[Cor.\,(2.6.25)]{Katz:p_adic_L-function_CM_fields}) and let $\Katzd^\kappa=\prod_{\sg\in\Sg}\Katzd(\sg)^{\kappa_\sg}$.
The following is a direct consequence of \propref{P:1.N}.
\begin{prop}\label{P:ESmeasure.V}There exists a $V(\plideal,\opcpt,\Zbarp)$-valued \padic measure $\EucE_\plideal$ on $\Glv$ such that
\begin{itemize}\item[(i)] $\EucE_\plideal $ is supported in $Z_1$,
\item[(ii)] for each $\hatads\in\frakX^+$, we have
\[\int_{\Glv}\hatads d \EucE_\plideal=\EucE_{\ads,\plideal}.\]
\end{itemize}
Moreover, if $\ads$ has infinity type $k\Sg+\kappa(1-c)$, then
\[\int_{\Glv}\hatads d \EucE_\plideal=\theta^\kappa\EucE_{\ads,\plideal}.\]
\end{prop}
\begin{proof}
Put $\bfa_\beta(\hatads,\frakc)=\iota_p(\beta^{\kappa}\bfa_\beta(\holES,\frakc))$. Recall that if $\ads$ is a Hecke character of infinity type $k\Sg+\kappa(1-c)$, then $\hatads_{\Sg_p}(\beta)=\iota_p(\beta^{k\Sg+\kappa})\ads_{\Sg_p}(\beta)$. By definition, we have \[\bfa_\beta(\hatads,\frakc)=\bfa^\setp_\beta(\ads,\frakc)\hatads_{\Sg_p}(\beta)\bbI_{\OFp^\x}(\beta).\] By the inspection of \eqref{E:FA.V}, we find that $\bda^\setp_\beta(\ads,\frakc)$ has the following form:
\beq\label{E:primetopFC.V}\begin{aligned}\bfa_\beta^\setp(\ads,\frakc)=&\sum_j b_j\cdot \hatads(a_j)\text{ for some } b_j\in\Zbarp\text{ and }\\
&a_j\in (\AFf^{(p\frakD)})^\x\prod_{v|\Csplit\Csplit^c} \OFv^\x\prod_{w|\cD_{\cK/\cF}\frakC^-}\cK^\x_w.\end{aligned}\eeq
Therefore, we have
\[\bfa_\beta(\hatads,\frakc)=\sum_{j}b_j\cdot \hatads((i_{\Sg_p}(\beta),a_j))\bbI_{\OFp^\x}(\beta),\quad i_{\Sg_p}(\beta)=(\beta,1)_p\in R_p^\x=(R_{\Sg_p}\oplus R_{\Sgbar_p})^\x.\]
For every $\phi\in\cC(\Glv,\Zbarp)$, we define $\bfa_\beta(\phi,\frakc):=\sum_{j}b_j\cdot \phi((i_{\Sg_p}(\beta),a_j))\bbI_{\OFp^\x}(\beta)$. Thus, $\phi\mapsto\bfa_\beta(\phi,\frakc)$ defines a $\Zbarp$-valued \padic measure on $\Glv$ supported in $Z_1$. Define a $\Zbarp\powerseries{(N^{-1}\frakc^{-1})_+}$-valued \padic measure $\EucE_\plideal(q)$ by
\[\int_{\Glv}\phi d\EucE_\plideal(q)=\sum_{\beta\in (N^{-1}\frakc^{-1})_+}\bfa_\beta(\phi,\plideal)q^\beta.\]
If $\hatads\in\frakX^+$, then $\kappa=0$ and $\bfa_\beta(\hatads,\frakc)=\iota_p(\bfa_\beta(\holES,\frakc))$, and we have
\[\int_{\Glv}\hatads d\EucE_\plideal(q)=\sum_{\beta\in(N^{-1}\frakc^{-1})_+}\bfa_\beta(\hatads,\frakc)q^\beta=\EucE_{\ads,\plideal}|_{(\OF,\frakc^{-1})}(q).\]
Therefore, by the $q$-expansion principle and the Zariski density of $\frakX^+$ in $\cC(\Glv,\Zbarp)$, the measure $\EucE_\plideal(q)$ descends to a unique $V(\plideal,\opcpt,\Zbarp)$-valued \padic measure $\EucE_\plideal$ supported in $Z_1$ such that
\[\int_{\Glv}\hatads d\EucE_\frakc|_{(\OF,\frakc^{-1})}(q)=\int_{\Glv}\hatads d\EucE_\frakc(q)\text{ for every }\hatads\in\frakX^+.\]

In addition, if $\ads$ has infinity type $k\Sg+\kappa(1-c)$, then
\beq\label{E:FC.V}\Katzd^\kappa\EucE_{\ads,\plideal}|_{(\OF,\frakc^{-1})}(q)=\sum_{\beta\in (N^{-1}\frakc^{-1})_+}\iota_p(\beta^\kappa \bfa_\beta(\holES,\frakc))q^\beta=\sum_{\beta\in (N^{-1}\frakc^{-1})_+}\bfa_\beta(\hatads,\frakc)q^\beta\eeq
by the effect of the \padic differential operator $\theta$ on the $q$-expansions \cite[(2.6.27)]{Katz:p_adic_L-function_CM_fields} (\cf\cite[\S 1.7 p.205]{HidaTilouine:KatzPadicL_ASENS}). This verifies the second assertion.\end{proof}
\def\holES{\bbE^h_{\ads}}
\subsection{The period integral}\label{SS:periodintegral.V}
We recall the period integral of the Eisenstein series calculated in \cite[\S 5]{Hsieh:Hecke_CM}. First we fix the choice of measures. For each finite place $v$ of $\cF$, let $\dx z_v$ be the normalized Haar measure on $\cK^\x_v$ so that $\vol(R^\x_v,\dx z_v)=1$ and let $\dx t_v=\dx z_v/\dx x_v$ be the quotient measure on $\cK^\x_v/\cF_v^\x$. If $v$ is archimedean, let $\dx t_v$ be the Haar measure on $\cK^\x_v/\cF_v^\x=\C^\x/\R^\x$ normalized so that $\vol(\C^\x/\R^\x,\dx t_v)=1$.
Let $\dx t=\prod'_v\dx t_v$ be the Haar measure on $\AK^\x/\AF^\x$ and let $\dx\bar{t}$ be the quotient measure of $\dx t$ on $\cK^\x\AF^\x\bksl
\AK^\x$ by the discrete measure on $\cK^\x$. Let $\phi_v=\Section$ if $v\in\bdh$ and $\phi_v=\Section^\nh$ if $v\in\bda$. Put \[l_{\cK_v}(\phi_v,\ads_v)=\int_{\cK_v^\x/\cF_v^\x}\phi_v(\rho(t)\cmptv)\ads_v(t)\dx t.\]Define the period integral $l_\cK(E^\nh_{\ads})$ of $E^\nh_\ads$ by
\[l_\cK(E^\nh_{\ads}):=\int_{\cK^\times \AF^\times\backslash
\AK^\times}E^\nh_\ads(\rho(t)\cmpt)\ads(t)\dx\bar{t}.\]
It is shown in \cite[\S 5.1]{Hsieh:Hecke_CM} that
\begin{align*}
l_\cK(E^\nh_\ads)=&\prod_vl_{\cK_v}(\phi_v,\ads_v)|_{s=0}.
\end{align*}

\begin{prop}\label{P:period.V}The local period integral $l_{\cK_v}(\phi_v,\ads_v)$ is given as follows:
\[
l_{\cK_v}(\phi_v,\ads_v)=\begin{cases} L(s,\ads_v)&\cdots\text{ $v\ndivides \frakD\cdot\infty$}, \\
L(s,\ads_v)\cdot \vol(\cK_v^\x/\cF_v^\x,\dx t)&\cdots\text{ $v\mid D_{\cK/\cF}\frakC^-$},\\
\vol(\C^\x/\R^\x,\dx t)&\cdots\text{ $v\mid \infty$}.\end{cases}
\]
\end{prop}
\begin{proof}These formulas can be found in \cite[\S 5.2, \S 5.3]{Hsieh:Hecke_CM}.
\end{proof}
It remains to determine the local period integral $l_{\cK_v}(\phi_v,\ads_v)$ for $v|p\Csplit\Csplit^c$.
In this case, $\phi_v$ is the Godement section $f_{\Phi^0_v,s}$ associated to the \BS function $\Phi^0_v$ defined in \eqref{E:sectionsplit.V}. We thus have
\begin{align*}
l_{\cK_v}(\phi_v,\ads_v)=&Z(s,\ads_v,\Phi_{\cK_v}):=\int_{\cK_v^\x}\Phi_{\cK_v}(z)\ads_v(z)\abs{z}_{\cK_v}^s\dx z,
\end{align*}
where $\Phi_{\cK_v}$ is given by
\[\Phi_{\cK_v}(z):=\Phi^0_v((0,1)\rho(z)\cmptv).\]
By a direct computation as in \cite[\S 5.2]{Hsieh:Hecke_CM}, we find that
\begin{align*}
Z(s,\ads_v,\Phi_{\cK_v})&=\ads_{\wbar}(-2\skewhf_w d_{\cF_v}^{-1})\ads_{w}(-2\skewhf_w)Z(s,\ads_{\wbar},\vp_{\wbar})Z(s,\ads_{w},\wh{\vp}_w)\quad(\,w|\Sg_p\Csplit).
\end{align*}
By Tate's local functional equation, we have
\begin{align*}Z(s,\ads_{w},\wh{\vp}_w)=&\frac{L(s,\ads_w)\ads_w(-1)}{\ep(s,\ads_w,\addchar)L(1-s,\ads_w^{-1})}.
\end{align*}
On the other hand, $Z(s,\ads_{\wbar},\vp_{\wbar})=1$ and $d_{\cF_v}=-2\skewhf_w$. Hence, we find that
\beq\label{E:2.V}l_{\cK_v}(\phi_v,\ads_v)=Z(s,\ads,\Phi_{\cK_v})=\ads_{w}(-2\skewhf_w)\cdot \frac{L(s,\ads_w)\ads_w(-1)}{\ep(s,\ads_w,\addchar)L(1-s,\ads_w^{-1})}\quad(\,w|\Sg_p\Csplit).\eeq

Define the modified Euler factors $Eul_p(\ads)$ and $Eul_{\frakC^+}(\ads)$ by
\beq\label{E:modifiedEuler.V}\begin{aligned}Eul_p(\ads):=&\prod_{w\in\Sg_p}Eul(\ads_w);\,Eul_{\frakC^+}(\ads)=\prod_{w|\Csplit}Eul(\ads_w),\text{ where }\\
&Eul(\ads_w):=\ads_w(2\skewhf_w)\cdot \frac{L(0,\ads_w)}{\ep(0,\ads_w,\addchar)L(1,\ads_w^{-1})}.\\
\end{aligned}\eeq
Combining \propref{P:period.V} and \eqref{E:2.V}, we obtain the following formula of the period integral $l_\cK(E^\nh_\ads)$.
\begin{prop}\label{P:periodintehral.V}Let $r$ be the number of prime factors of $D_{\cK/\cF}$. We have
\[l_\cK(E^\nh_\ads)=2^{r}\cdot L^{(p\frakC)}(0,\ads)\cdot Eul_p(\ads)Eul_{\frakC^+}(\ads).\]
\end{prop}

\subsection{Katz \padic $L$-functions.}\label{SS:KatzpadicL.V}
Let $\CLKF=\cK^\x\AFf^\x\bksl \AKf^\x/(\OK\ot_\Z\Zhat)^\x$. Recall that we introduced a subgroup $\cT$ of $\AKf^\x$ in \remref{R:1.V}. Let $\CLKF^\alg$ be the subgroup of $\CLKF$ generated by the image of $\cT$ in $\CLKF$. It is easy to see that $\CLKF^\alg$ is in fact generated by primes ramified over $\cF$. In particular, $\#\CLKF^\alg$ is a power of $2$. Let $\cD_1$ be a set of representatives of $\CLKF/\CLKF^\alg$ in $(\AKf^{(\frakD)})^\x$. Following \cite[(4.12)]{Hida:mu_invariant}, we let $\EucL_{\frakC,\Sg}$ be the \padic measure on $\Glv$ such that for each $\phi\in\cC(\Glv,\Zbarp)$,
\[\int_{\Glv}\phi d\,\EucL_{\frakC,\Sg}=\sum_{a\in\cD_1}\int_{\Glv}\phi|[a] d\EucE_{\frakc(a)}(x(a)).\]
Here $\phi|[a]\in\cC(\Glv,\Zbarp)$ is the translation given by $\phi|[a](x):=\phi(x\rec_\cK(a))$. By \propref{P:ESmeasure.V}, if $\wh\ads$ is the \padic avatar of a Hecke character $\ads$ of infinity type $k\Sg+\kappa(1-c)$, then we have
\beq\label{E:53.N}\int_{\Glv}\hatads d\,\EucL_{\frakC,\Sg} =\sum_{a\in\cD_1}\ads(a)\theta^\kappa\EucE_{\ads,\frakc(a)}(x(a)).\eeq

Let $\cU$ be the torsion subgroup of $\cK^\x$ and let $\cU^\alg=(\cK^\x)^{1-c}\cap \OK^\x$ be a subgroup of $\cU$. We have the following evaluation formula of the measure $\EucL_{\frakC,\Sg}$.
\begin{prop}\label{P:EVCM} Let $(\Omega_\infty,\Omega_p)\in(\C^\x)^\Sg\x(\Zbarp^\x)^\Sg$ be the complex and \padic CM periods of $(\cK,\Sg)$ respectively. Then we have
\begin{align*}\frac{1}{\Omega_p^{k\Sg+2\kappa}}\cdot\int_{\Glv}\hatads d\,\EucL_{\frakC,\Sg}&=L^{(p\frakC)}(0,\ads)\cdot Eul_p(\ads)Eul_{\frakC^+}(\ads)\\
&\times \frac{\pi^{\kappa}\Gamma_\Sg(k\Sg+\kappa)}{\sqrt{\abs{D_\cF}_\R}(\Im \skewhf)^\kappa\cdot \Omega_\infty^{k\Sg+2\kappa}}\cdot[\OK^\x:\OF^\x]\cdot t_\cK,\end{align*}
where\[t_\cK=\frac{\# \cU^\alg }{[\OK^\x:\OF^\x]}\cdot\frac{2^r}{\#\CLKF^\alg}.\]
Note that $t_\cK$ is a power of $2$.
\end{prop}
\begin{proof}
Let $\delta^\kappa_k$ be the Maass-Shimura differential operator on modular forms of weight $k$ (See \cite[(1.21)]{HidaTilouine:KatzPadicL_ASENS}). By \cite[(4.22), (5.2)]{Hsieh:Hecke_CM} we find that
\begin{align}\label{E:41.N}\delta_k^\kappa \bbE^h_{\ads}&=\frac{1}{(-4\pi)^{\kappa}}\cdot\frac{\Gamma_\Sg(k\Sg+\kappa)}{\Gamma_\Sg(k\Sg)}\cdot
\bbE^\nh_{\ads},\\
\label{E:diffCM}\frac{1}{\Omega_p^{k\Sg+2\kappa}}\cdot \theta^\kappa\EucE_{\ads,\frakc(a)} (x(a))&=\frac{(2\pii)^{k\Sg+2\kappa}}{\Omega_\infty^{k\Sg+2\kappa}}\cdot \delta^\kappa_k\bbE^h_{\ads,\frakc(a)}(x(a)),\,a\in(\AKf^{(\frakD)})^\x.\end{align}
Let $U_\cK= (\C_1)^\Sg\x(\OK\ot\Zhat)^\x$ be an open-compact subgroup of $\AK^\x=(\C^\x)^\Sg\x\AKf^\x$, where $\C_1$ is the unit circle in $\C^\x$. Let $\ol{U}_\cK$ denote the image of $U_\cK$ in $\cK^\x\AF^\x\bksl \AK^\x$. Then \eqref{E:53.N} equals
\begin{align*}
\frac{1}{\Omega_p^{k\Sg+2\kappa}}\cdot\int_{\Glv}\hatads d\,\EucL_{\frakC,\Sg}
&=\frac{(2\pii)^{k\Sg+2\kappa}}{\Omega_\infty^{k\Sg+\kappa}}\cdot \sum_{a\in\cD_1}\ads(a)\delta^\kappa_k\bbE^h_{\ads,\frakc(a)}(x(a))& \text{ by }\eqref{E:diffCM}\\
&=\frac{(2\pii)^{k\Sg+2\kappa}}{\Omega_\infty^{k\Sg+\kappa}}\cdot\frac{1}{\#\CLKF^\alg\cdot \vol(\ol{U}_\cK,\dx \bar{t})}\int_{\cK^\x\AF^\x\bksl \AK^\x}\ads(t)\delta^\kappa_k\bbE^h_{\ads}(x(t))\dx \bar{t}& \text{ by }\eqref{E:rightinv.V}\\
&=\frac{\pi^{\kappa}\Gamma_\Sg(k\Sg+\kappa)}{\sqrt{\abs{D_\cF}_\R}\Im(\skewhf)^\kappa\cdot\Omega_\infty^{k\Sg+2\kappa}}\cdot\frac{\#\cU^{\alg}}{\#\CLKF^\alg}\cdot l_\cK(E^\nh_{\ads}).&\text{ by }\eqref{E:ES1.N},\,\eqref{E:41.N}.
\end{align*}
It is clear that the proposition follows from \propref{P:periodintehral.V}.
\end{proof}
\begin{Remark} The evaluation formula for the measure $\EucL_{\frakC,\Sg}$ in \propref{P:EVCM} agrees with the measure $\vp^*$ constructed in \cite[Thm.\,4.2]{HidaTilouine:KatzPadicL_ASENS} up to a product of local Gauss sums at $v|\Csplit\Csplit^c$ and $t_\cK$, both of which are \padic units.\end{Remark}
\def\lp{\eta_p}
\section{Hida's theorem on the anticyclotomic $\mu$-invariant}\label{S:Hida.V}
\def\Wm{\baseR_m}
\def\Eau{\Eadsu}
\renewcommand\Dmd[1]{\left<{#1}\right>_\Sg} 
\subsection{}
We fix a Hecke character $\adsx$ of infinity type $k\Sg$, $k\geq 1$ and suppose $\frakC$ is the prime-to-$p$ conductor of $\adsx$. Let $\Glv^-$ be the anticyclotomic quotient of $\Glv$. We have an isomorphism \[\rec_\cK:\prolim_n\cK^\x\AFf^\x\bksl\AKf^\x/U_\cK(\frakC p^n)\isoto \Glv^-.\] Let $\Gamma^-$ be the maximal $\Zp$-free quotient of $\Glv^-$.
Each function $\phi$ on $\Gamma^-$ will be regarded as a function on $\Glv$ by the natural projection $\pi_-\colon\Glv\to \Glv^-\to \Gamma^-$. We define the anticyclotomic projection $\EucL_{\adsx,\Sg}^-$ of the measure $\EucL_{\frakC,\Sg}$ by
\[\int_{\Gamma^-} \phi d\EucL_{\adsx,\Sg}^-:=\int_{\Glv}\wh\adsx\phi d\,\EucL_{\frakC,\Sg}.\]
In what follows, we introduce an open subgroup $\Gamma'$ of $\Gamma^-$ and compute the $\mu$-invariant of $\EucL_{\frakC,\Sg}$ restricted to $\Gamma'$. The introduction of $\Gamma'$ is to treat the case the minus part $h_\cK^-$ of the class number of $\cK$ is divisible by $p$.

Let $\Gamma'$ be the open subgroup of $\Gamma^-$ generated by the image of $\cZ_1$ in \eqref{E:A} and let $Z':=\pi_-^{-1}(\Gamma')$ be the subgroup of $Z(\frakC)$. Then we have $Z'\supset Z_1$. In addition, the reciprocity law $\rec_\cK$ at $\Sg_p$ induces an injective map $\rec_{\Sg_p}\colon 1+p\OFp\hookto \OFp^\x=\OK_{\Sg_p}^\x\stackrel{\rec_\cK}\longto Z(\frakC)^-$ with finite cokernel as $p\ndivides D_{\cF}$, and this map $\rec_{\Sg_p}$ induces an isomorphism $\rec_{\Sg_p}:1+p\Op\isoto\Gamma'$. We thus identify $\Gamma'$ with the subgroup $\rec_{\Sg_p}(1+p\OFp)$ of $Z(\frakC)^-$. Note that $\Gamma'=\Gamma^-\iso 1+p\OFp$ if $p\ndivides h_\cK^-$. Let $Cl'_-\supset Cl^\alg_-$ be the image of $Z'$ in $Cl_-$ and let $\cD'_1$ (resp. $\cD_1''$) be a set of representatives of $Cl'_-/Cl^\alg_-$ (resp. $Cl_-/Cl'_-$) in $(\AKf^{(\frakD)})^\x$ (so $\cD''=\stt{1}$ if $p\ndivides h_\cK^-$). Let $\cD_1:=\cD_1''\cD'_1$ be a set of representatives of $Cl_-/Cl^\alg_-$. For each $b\in\cD_1''$, we denote by $\EucL^b_{\adsx,\Sg}$ the \padic measure on $1+p\OFp\iso\Gamma'$ obtained by the restriction of $\EucL_{\adsx,\Sg}^-$ to $b.\Gamma':=\pi_-(\rec_\cK(b))\Gamma^-$. To be precise, we have
\beq\label{E:12.V}\begin{aligned}\int_{\Gamma'}\phi d\EucL^b_{\adsx,\Sg}:=&\int_{\Gamma^-}\bbI_{b.\Gamma'}\cdot(\wh\adsx\phi)|[b^{-1}] d\EucL_{\adsx,\Sg}^-\\
=&\sum_{a\in b\cD_1'}\int_{Z(\frakC)}\bbI_{Z'}\cdot (\wh\adsx\phi)|[ab^{-1}] d\EucE_{\frakc(a)}(x(a))\\
=&\sum_{a\in b\cD_1'}\adsx(ab^{-1})\int_{Z(\frakC)}\wh\adsx\cdot \phi|[ab^{-1}] d\EucE_{\frakc(a)}(x(a)),\end{aligned}\eeq
where $\bbI_{b.\Gamma'}$ and $\bbI_{Z'}$ are the characteristic functions of $b.\Gamma'$ and $Z'$. Note that the last equality follows from the fact that the Eisenstein measure $\EucE_{\frakc(a)}$ has support in $Z_1\subset Z'$ (\propref{P:ESmeasure.V} (i)). Recall that the $\mu$-invariant $\mu(\vp)$ of a $\Zbarp$-valued \padic measure $\vp$ on a \padic group $H$ is defined to be
\[\mu(\vp)=\inf_{U\subset H\text{ open }} v_p(\vp(U)).\]
Let $\Imu^-$ and $\Imu^b$ denote the Iwasawa $\mu$-invariants of the \padic measures $\EucL_{\adsx,\Sg}^-$ and $\EucL^b_{\adsx,\Sg}$ respectively.
\begin{lm}\label{L:4.V}We have $\Imu^-=\inf\limits_{b\in\cD_1''}\Imu^b$.\end{lm}
\begin{proof}This is clear from the definitions of $\mu$-invariants and $\Gamma^-=\disjoint_{b\in\cD_1''}b.\Gamma'$ is a disjoint union.\end{proof}
We shall follow Hida's approach to compute the $\mu$-invariants $\Imu^b$ via an explicit calculation of the Fourier coefficients of the Eisenstein series, using a deep result on the linear independence of modular forms modulo $p$ \cite[Cor.\,3.21]{Hida:mu_invariant}.
\subsection{}
Fix $\frakc=\frakc(\OK)$. A functorial point in $\Ig_K(\frakc)$ will be written as $(\ulA,\eta)$, where $\ulA=(A,\lam,\iota)$ and $\eta=(\ol{\eta}^\setp,\lp)$. Enlarging $\CMring$ if necessary, we let $\CMring$ be the \padic ring generated by the values of $\lam$ on finite ideles over $W$. Let $\frakm_\CMring$ be the maximal ideal of $\CMring$ and fix an isomorphism $\CMring/\frakm_\CMring\isoto\Fpbar$.
Let $T:=\OF^*\ot_\Z\bbmu_{p^\infty}$ and let $\wh T=\dirlim_m T_{/\CMring/\frakm_\CMring^m}=\OF^*\ot_\Z\formal{\bbG}_m$. Let $\stt{\xi_1\cdots,\xi_d}$ be a basis of $\OF$ over $\Z$ and let $t$ be the character $1\in \OF=X^*(\OF^*\ot_\Z\Gm)=\Hom(\OF^*\ot_\Z\Gm,\Gm)$. Then we have $\cO_{\wh T}\isoto \CMring\powerseries{t^{\xi_1}-1,\cdots t^{\xi_d}-1}$. For $y=(\ul{A}_y,\eta_y)\in\Ig_K(\frakc)(\Fpbar)\subset \Ig_K(\Fpbar)$, it is well known that the deformation space $\wh S_{y}$ of $y$ is isomorphic to the formal torus $\wh T$ by the theory of Serre-Tate coordinate (\cite{Katz:ST}). The $p^\infty$-structure $\eta_{y,p}$ of $A_y$ induces a canonical isomorphism $\vphi_{y}\colon\wh{T}\isoto \wh{S}_{y}=\Spf\wh\cO_{\Ig_K(\frakc),y}$ (\cf \cite[(3.15)]{Hida:mu_invariant}).

Now let $\bfx:=x(1)_{/\CMring}\in\Ig_K(\frakc)(\CMring)$ be a fixed CM point of type $(\cK,\Sg)$ and let $x_0=\bfx\ot_{\CMring}\Fpbar=(\ul{A}_0,\eta_0)$. For a deformation $z=(\ul{A},\eta)_{/\cR}\in \wh S_{x_0}(\cR)$ of $x_0$ over an artinian local ring $\cR$ with the maximal ideal $\frakm_\cR$ and the residue field $\Fpbar$, we let $t(\ul{A},\eta):=t(\vphi_{x_0}^{-1}((\ul{A},\eta)_{/\cR}))\in 1+\frakm_\cR$. Then $\bfx$ is the canonical lifting of $x_0$, \ie $t(\bfx)=1$. For $f\in V(\frakc,K,\CMring)$, we define \[f(t):=\vphi_{x_0}^*(f)\in\cO_{\wh T}=\CMring\powerseries{T_1,\cdots T_d}\quad(T_i=t^{\xi_i}-1).\]
We call the formal power series $f(t)$ the \emph{$t$-expansion} around $x_0$ of $f$. For each $u\in\OFp^\x$, let $uz:=(\ul{A},\ol{\eta}^\setp,u\lp)$ be a deformation of $ux_0$. Then we have $t(uz)=t(z)^u$ and hence $\vphi_{ux_0}^*(f)(t)=\vphi_{x_0}^*(f)(t^u)=f(t^u)$.

For each $a\in \cD_1'$, let $\Dmd{a}$ be the unique element in $1+p\Op$ such that $\rec_{\Sg_p}(\Dmd{a})=\pi_-(\rec_\cK(a))\in\Gamma'$. Recall that $\cU_p$ is the torsion subgroup of $\Op^\x$. For every pair $(u,a)\in\cU_p\x \cD_1$, we write $\bbE_{u,a}$ for $\Eadsu|_{\frakc(a)}\in\cM_k(\frakc(a),K,\bR)$ and let $\EucE_{u,a}$ be the \padic avatar of $\bbE_{u,a}$. Fix a sufficient large finite extension $L$ over $\Qp$ so that $\adsx$ and $\bbE_{u,a}|[a]$ are defined over $\cO_L$ for all $(u,a)$, and hence $\EucE_{u,a}|[a]\in V(\frakc,\cO_L)$. For $(a,b)\in\cD_1\x\cD_1''$, we define
\begin{align*}\wtd\cE_a(t)&=\sum_{u\in \torsbgp}\EucE_{u,a}(t^{u^{-1}}),\\
\cE^b(t)&=\sum_{a\in b\cD_1'}\adsx(ab^{-1})\wtd\cE_a|[a](t^{\Dmd{ab^{-1}}}).\end{align*}
For $E=\bbE_{u,a}$, $\EucE_{u,a}(t)$ or $\cE^b(t)$, we define $\mu(E)\in\Q_{\geq 0}$ by
\[\mu(E)=\inf\stt{v_p(\uf_L^m)\mid \uf_L^{-m}E\not\con 0\pmod{\frakm_L}\quad(m\in\Z_{\geq 0})}.\]
\begin{prop}\label{P:2.V}
The formal power series $\cE^b(t)$ equals the power series expansion of the measure $\EucL^b_{\adsx,\Sg}$ regarded as a \padic measure on $\OFp$ supported on $1+p\OFp$. In particular, we have $\mu(\cE^b(t))=\Imu^b$.
\end{prop}
\begin{proof}
We compute the $t$-expansion of $\cE^b$. For $\kappa\in\Z_{\geq 0}[\Sg]$, let $\nu_\kappa$ be the \padic character of $\Gamma'$ such that $\nu_\kappa(\rec_{\Sg_p}(y))=y^\kappa,\,y\in1+p\Op$. By the definition of $\cE^b$, we find that
\begin{align*}\Katzd^\kappa\cE^b|_{t=1}&=\sum_{a\in b\cD_1'}\adsx(ab^{-1})\Dmd{ab^{-1}}^\kappa\Katzd^\kappa\wtd\cE_a|[a]|_{t=1}
=\sum_{a\in b\cD_1'}\adsx\nu_\kappa(ab^{-1})\theta^\kappa\EucE_{\cU,a}(x(a)),
\end{align*}where $\EucE_{\cU,a}:=\sum_{u\in\torsbgp}u^{-\kappa}\EucE_{u,a}$. Let $\adsx_\kappa$ be the Hecke character such that the \padic avatar $\wh\adsx_\kappa$ is $\wh\adsx \nu_\kappa$. Then $\adsx_\kappa$ has infinity type $k\Sg+\kappa(1-c)$. We are going to show that $\theta^\kappa\EucE_{\cU,a}=\theta^\kappa\EucE_{\adsx_\kappa,\frakc(a)}$ by comparing the $q$-expansions. A key observation is that since $\nu_\kappa$ is anticyclotomic and unramified outside $p$, we find that $\bfa_\beta^\setp(\adsx,\frakc(a))=\bfa_\beta^\setp(\adsx\nu_\kappa,\frakc(a))$ in view of \eqref{E:primetopFC.V}. By the inspection of the $q$-expansion of $\EucE_{\cU,a}$ at $(\OF,\frakc(a)^{-1})$, we find that
\begin{align*}
\theta^\kappa\EucE_{\cU,a}(q)&=\sum_{u\in\torsbgp}\sum_{\beta\in\cF_+}u^{-\kappa}\bfa_\beta^\setp(\adsx,\frakc(a))\adsx_{\Sg_p}(\beta)\bbI_{u(1+p\OF_p)}(\beta)\beta^{k\Sg+\kappa} q^\beta\\
&=\sum_{\beta\in\cF_+}\bfa_\beta^\setp(\adsx,\frakc(a))\bbI_{\Op^\x}(\beta)\hatads_{\Sg_p}(\beta)\Dmd{\beta}^\kappa q^\beta\\
&=\sum_{\beta\in\cF_+}\bfa_\beta^\setp(\adsx\nu_\kappa,\frakc(a))\bbI_{\Op^\x}(\beta)\wh\adsx_{\Sg_p}\nu_\kappa(\beta) q^\beta\\
&=\sum_{\beta\in\cF_+}\bfa_\beta(\wh\adsx\nu_\kappa,\frakc(a))q^\beta=\theta^\kappa\EucE_{\adsx_\kappa,\frakc(a)}(q)\quad \text{by }\eqref{E:FC.V}.
\end{align*}
We thus conclude that $\theta^\kappa\EucE_{\cU,a}=\theta^\kappa\EucE_{\adsx_\kappa,\frakc(a)}$ by the $q$-expansion principle. By \eqref{E:12.V}, we have
\begin{align*}\Katzd^\kappa\EucE^b|_{t=1}=&\sum_{a\in b\cD_1'}\adsx\nu_\kappa(ab^{-1})\theta^\kappa\EucE_{\adsx_\kappa}(x(a))\\
=&\sum_{a\in b\cD_1'}\int_{Z(\frakC)}\wh\adsx\nu_\kappa|[ab^{-1}]d\EucE_{\frakc(a)}(x(a))=\int_{\Gamma'}\nu_\kappa d\EucL^b_{\adsx,\Sg}.\end{align*}
In other words, $\Katzd^\kappa\cE^b|_{t=1}$ interpolates the $\kappa$-th moment of the measure $\EucL^b_{\adsx,\Sg}$, and hence the proposition follows.
\end{proof}
\begin{remark}If $p$ does not divide $h_\cK^-$, the $t$-expansion of $\cE^1$ is the power series expansion of the \padic $L$-function $\EucL^-_{\adsx,\Sg}$.
\end{remark}
\subsection{}
 Let $\Dmd{\cD_1'}\subset 1+p\OFp$ be the image of $\cD_1'$ under $\Dmd{\cdot}$. Regarding $\cU^\alg$ as a subgroup of $\cU_p$ by the imbedding induced by $\Sg_p$, we let $\cD_0$ be a set of representatives of $\cU_p/\cU^\alg$ in $\cU_p$.
\begin{lm}\label{L:2.V}Put $\bfD:=\cD_0\Dmd{\cD_1'}\subset \Op^\x$. Then the quotient map $\bfD\to\Op^\x/(\OK_\setp^\x)^{1-c}$ is injective.
\end{lm}
\begin{proof} Let $a_1,a_2\in\cD_1'$ and $u_1,u_2\in\torsbgp$. Let $a=a_1a_2^{-1}$ and $\fraka=\il_\cK(a)$.
Suppose that $u_1\Dmd{a_1}=u_2\Dmd{a_2}\al^{1-c}$ for some $\al\in\OK_\setp^\x$. Let $y=\al^{-1}a\in\AKf^\x$ and $[y]:=\rec_\cK(y)\in \Glv^-$. Then it is easy to see that $[y^2]=[yy^{-c}]$ is in the torsion subgroup $\Delta$ of $\Glv^-$, and hence $[y]=[\al^{-1}a]\in\Delta$. It follows that the ideal $(\al)^{-1}\fraka$ is a product of ramified primes and ideals of $\OF$. So $[\fraka]\in\CLKF^\alg$ and $a_1=a_2$.
\end{proof}
\begin{lm}[Prop.\,3.4 \cite{Hida:mu_invariant}] \label{L:3.V}Let $\al\in\OK_\setp^\x\subset \AKf^\x$ and let $f\in\cM_{k}(\frakc(\al),K,\Zbarp)$. We have
\[\wh f(t^{\al^{1-c}})=\al^{-k\Sg}\cdot \wh f|[\al](t).\]
\end{lm}
\begin{proof} Since $\al\in\OK_\setp^\x$, we can find $x_0\smid [\al]:=(\ulA_0^\al,\eta^\al_0)\in\Ig_K(\plideal(\al))(\Fpbar)$ in the prime-to-$p$ isogeny class $[(\ul{A}_0,\rho_\cmpt(\al)\eta)]\in \cI_{K,n}^\setp(\Fpbar)$ $(\rho_\cmpt(\al):=\cmpt^{-1}\rho(\al)\cmpt)$ together with a prime-to-$p$ isogeny $\xi_{\al}:x_0\smid [\al]=(\ulA^\al_0,\eta^\al_0)\to x_0=(\ul{A}_0,\eta)$. Then $\xi_\al$ induces an isomorphism $\tilde{\xi}_\al\colon\wh S_{x_0}\isoto \wh S_{x_0\smid [\al]}$, which sends a deformation $\ul{A}_{/\cR}$ of $x_0$ over a local artinian ring $\cR$ to the deformation $\ul{A}^\al{}_{/\cR}$ of $x_0\smid [\al]$. In addition, there exists a unique isogeny $\xi_{\al,\cR}:(\ul{A}^\al,\eta^\al)\to (\ul{A},\eta)$ with the following commutative diagram:
\[\xymatrix{\OFp^*\ot_\Z\bbmu_{p^\infty}\ar[d]^{\al^{\Sgbar}}\xyinj^{\qquad \lp^\al}&A^\al[p^\infty]\ar[d]^{\xi_{\al,\cR}}\xysurj^{(\lp^\al)^{(-1)}}& \OFp\ot\Qp/\Zp\ar[d]^{\al^\Sg}\\
\OFp^*\ot_\Z\bbmu_{p^\infty}\xyinj^{\qquad \lp}&A[p^\infty]\xysurj^{\lp^{(-1)}}& \OFp\ot\Qp/\Zp.}\]
Here $\lp^{(-1)}$ and $(\lp^\al)^{(-1)}$ are morphisms induced by $\lp$ and $\lp^\al$ together with the polarizations of $A$ and $A^\al$ via Cartier duality. Therefore, we find that \[t(\ul{A}^\al,\eta^\al)=t(\ul{A},\eta)^{\al^{1-c}},\]
and that
\begin{align*}f|[\al](\ulA,\eta,\Om(\lp))&=f(\ul{A},\rho_\cmpt(\al)\eta,\Om(\lp))\\
&=f(\ul{A}^\al,\eta^\al,\xi_{\al,\cR}^*\Om(\lp))\\
&=f(\ul{A}^\al,\eta^\al,\al^{-\Sg}\Om((\lp^\al)))=\al^{k\Sg}\cdot f(\ul{A}^\al,\eta^\al,\Om((\lp^\al))).
\end{align*}
It is clear that
\[\wh f(t^{\al^{1-c}})=\al^{-k\Sg}\cdot \wh f|[\al](t).\qedhere\]
\end{proof}

The following theorem is due to Hida \cite[Thm.\,5.1]{Hida:mu_invariant}.
\begin{thm}[Hida]\label{T:1.V}The $\mu$-invariant $\Imu^-$ is given by the following formula
\[\Imu^-=\inf_{\substack{(u,a)\in\cD_0\x\cD_1,\\\beta\in\cF_+}}v_p(\bfa_\beta(\Eadsu,\frakc(a))).\]
\end{thm}
\begin{proof}
Let $v\in\cU^\alg$. We may write $v^{-1}=\al^{1-c}$ for some $\al\in\OK_\setp^\x$ since $p$ is assumed to be unramified in $\cK$. Regarding $\al$ as an idele in $\AK^\x$, we denote by $\al_\infty$ and $\al_f$ the infinite and finite components of $\al$ respectively. Let $(a,b)\in\cD_1\x \cD_1''$. By \lmref{L:3.V} and \eqref{E:rightinv.V}, for each $u\in\cU_p$ we have
\[\EucE_{uv,a}(t^{u^{-1}v^{-1}})=\al^{-k\Sg}\EucE_{uv,a}|[\al_f](t^{u^{-1}})=\EucE_{u,a}(t^{u^{-1}})\lam^{-1}(\al_\infty)\adsx^{-1}(\al_f)=\EucE_{u,a}(t^{u^{-1}}).\]
Therefore, we find that
\begin{align*}\wtd\cE_a(t)&=\#\cU^\alg\cdot \sum_{u\in\torsbgp/\cU^\alg}\EucE_{u,a}(t^{u^{-1}})\\
\intertext{and }
\cE^b(t)&=\sum_{a\in b\cD_1'}\adsx(ab^{-1})\sum_{u\in \torsbgp}\EucE_{u,a}|[a](t^{\Dmd{ab^{-1}}u^{-1}})\\
&=\#\cU^\alg\cdot \sum_{(u,a)\in\cD_0\x b\cD_1'}\adsx(ab^{-1})\EucE_{u,a}|[a](t^{\Dmd{ab^{-1}}u^{-1}}).\end{align*}
Note that $p\ndivide\#\cU^\alg$ as $\cU^\alg$ is a subgroup of the torsion subgroup in $\cK^\x$ and $p\ndivide 2\cdot D_\cF$.
From \lmref{L:2.V} and the linear independence of modular forms modulo $p$ \cite[Thm.\,3.20,\,Cor.\,3.21]{Hida:mu_invariant}, we deduce that \[\mu(\cE^b(t))=\inf_{(u,a)\in\cD_0\x b\cD_1'}\mu(\EucE_{u,a}(t)).\]
Since $\EucE_{u,a}|[a]$ is the \padic avatar of $\bbE_{u,a}$, it follows from the irreducibility of Igusa tower that $\mu(\EucE_{u,a}|[a](t))=\mu(\EucE_{u,a}|[a])=\mu(\bbE_{u,a})$. From the $q$-expansion principle of \padic modular forms (\cite{DR_padic_L}) and \lmref{L:4.V}, we find that
\[\Imu^-=\inf_{b\in\cD_1''}\mu(\cE^b(t))=\inf_{\substack{(u,a)\in\cD_0\x\cD_1,\\\beta\in\cF_+}}v_p(\bfa_\beta(\Eadsu,\frakc(a))).\qedhere\]
\end{proof}

\begin{cor}\label{C:2.V}Suppose that
\begin{itemize}
\item[(L)] $\mu_p(\adsx_v)=0$ for every $v|\frakC^-$,
\item[(N)]$\adsx$ is not residually self-dual, namely $\wh{\adsx}_+\not\con \qchKF\om_\cF\pmod{\frakm}$.
\end{itemize}
Then $\Imu^-=0$.
\end{cor}
\begin{proof} It follows from \cite[Prop.\,6.3 and Lemma 6.4]{Hsieh:Hecke_CM} (following an argument of Hida) that if $\adsx$ is not residually self-dual, then for some $a\in\cD_1$ we can find $\beta\in\cO_{\cF,\setp}^\x$ such that
\[\bfa_\beta^\setp(\adsx,\frakc(a)))\not\con 0\pmod{\frakm}\iff \bfa_\beta(\bbE^h_\adsx,\frakc(a))=\sum_{u\in\cU_p}\bfa_\beta(\Eadsu,\frakc(a))\not\con 0\pmod{\frakm}.\]
Hence, $v_p(\bfa_\beta(\Eadsu,\frakc(a)))=0$ for $u\con \beta\pmod{p}$. We conclude that $\Imu^-=0$ by \thmref{T:1.V}.
\end{proof}

\def\OKbasis{\bftheta}
\def\twobeta{\beta}
\def\Schar{\xi}
\section{Proof of Theorem A}\label{S:ThetaIntegral.V}
\subsection{}
In this subsection, we fix a place $v|\frakC^-$ and let $w$ be the place of $\cK$ above $v$. Let $E=\cK_v$ and $F=\cF_v$. Let $\nads_v:=\adsx_v\Abs_E^{-\onehalf}$ be a character of $E^\x$ such that $\nads_v|_{F^\x}=\tau_{E/F}$. Let $d_F$ be the generator of $\cD_\cF$ and $\delta=\delta_v=2d_F^{-1}\skewhf$ be the generator of the different $\cD_{E/F}$ fixed in \subsecref{S:different}.
\begin{lm}\label{L:key.V}Let $\beta\in F^\x$. If $A_\beta(\adsx_v)\not =0$, then
\[W(\nads_v)\tau_{E/F}(\beta)=\nads_v(2\skewhf).\]
\end{lm}
\begin{proof}
The idea is to identify $A_\beta(\adsx)$ with the Whittaker integrals of a certain Siegel-Weil section on $U(1,1)$.
Let $W=E$ with the skew-Hermitian form $\pair{x}{y}_W=\delta x\ol{y}$.
Let $G=U(W)(F)$ be the associated unitary group and let $H=U(W+W^-)(F)$, where $W^-$ is the Hermitian space $(W,-\pairing_W)$. We let $\Schar=\nads_v$ and define the induced representation $\bfI(\Schar,s)$ of $H$ by
\[\bfI(\Schar,s)=\stt{\text{smooth }f:H\to\C\mid f(\MX{a}{b}{0}{\ol{a}^{-1}}h)=\Schar(a)\abs{a}_E^{s+\onehalf}f(h)}.\]
Let $\Delta=\delta^2\in F$ and let $T=-\Delta\beta d_F^{-1}$. For $f\in\bfI(\Schar,s)$, following \cite[(6.5) p.969]{Harris:Theta_dichotomy} we define the Whittaker integral by
\[\cW_T(s)(f):=\int_F f(\MX{0}{-1}{1}{0}\MX{1}{x}{0}{1})\psi(T x)dx.\]
We embed $G$ into $H$ by
\[g\mapsto i(g,1)=\onehalf\MX{g+1}{\frac{1}{2\delta}(g-1)}{2\delta(g-1)}{g+1}.\]
Let $\Phi_{\bfone}$ be the section $\Phi_\eta$ defined in \cite[p.989 (8.6)]{Harris:Theta_dichotomy} with $\eta=\bfone$ the trivial character.
Then $\Phi_{\bfone}$ is the unique function in $\bfI(\Schar,s)$ such that $\Phi_{\bfone}(1)=1$ and $\Phi_{\bfone}(h i(g,1))=\Phi_{\bfone}(h)$ for every $g\in G$.
Recall that
\[A_\beta(\adsx_v)=\int_{F^\x}\Schar^{-1}(x+2^{-1}\delta)\abs{x+\delta}_E^{-\onehalf}\psi(-\beta d_F^{-1}x)dx.\]
By the calculation in \cite[p.990 (8.14)]{Harris:Theta_dichotomy}, we find that
\beq\label{E:3.V}\begin{aligned}\cW_T(0)(\Phi_\bfone)&=\Schar(-1)\cdot\int_F\Schar^{-1}\Abs_E^{-s-\onehalf}(x+\frac{1}{2\delta})\psi(-\Delta \beta d_F^{-1} x)dx|_{s=0}\\
&=\Schar\Abs_E^{s-\onehalf}(-2\Delta)\int_F\Schar^{-1}\Abs_E^{-s-\onehalf}(x+2^{-1}\delta)\psi(-\beta d_F^{-1}x)dx|_{s=0}\\
&=\Schar\Abs_E^{-\onehalf}(-2\Delta)\cdot A_\beta(\adsx_v).\end{aligned}
\eeq
On the other hand, let $\chi^{-c}(z):=\chi^{-1}(\zbar)$ and let $M^*(s,\Schar):\bfI(s,\Schar)\to \bfI(-s,\Schar^{-c})$ be the normalized intertwining operator defined in \cite[(6.8)]{Harris:Theta_dichotomy}. By \cite[(6.10) and Cor.\,8.3 (ii)]{Harris:Theta_dichotomy}, we have
\begin{align*}
\cW_T(0)(M^*(0,\Schar)\Phi_{\bfone})&=\tau_{E/F}(T)\cdot \cW_T(0)(\Phi_{\bfone});\\
M^*(0,\Schar)\Phi_{\bfone}&=\Schar(-\delta)\cdot W(\Schar)\cdot \Phi_{\bfone}.
\end{align*}
Therefore, it follows from \eqref{E:3.V} that \[\tau_{E/F}(-\Delta\beta d_F^{-1})\cdot A_\beta(\adsx_v)=\Schar(-\delta)W(\Schar)\cdot A_\beta(\adsx_v),\] and hence
\[ A_\beta(\adsx_v)\not =0\imply \tau_{E/F}(\beta d_F)=\Schar(-\delta)W(\Schar)\iff W(\Schar)\tau_{E/F}(\beta)=\Schar(2\skewhf).\qedhere\]
\end{proof}
\begin{prop}\label{P:6.Theta}  Suppose $v$ is inert and $w(\frakC^-)=1$. We have $v_p(A_\beta(\adsx_v))\geq \mu_p(\adsx_v)$ for all $\beta\in F^\x$. In addition, there exists $\Beth_v\in \uf^{-1}\cO_F^\x$ such that $v_p(A_{\Beth_v}(\adsx_v))=\mu_p(\adsx_v)$.
\end{prop}
\begin{proof}
Let $\stt{1,\OKbasis}=\stt{1,\bftheta_v}$ be the $\OFv$-basis of $\OKv$ fixed in \subsecref{S:different} such that $\UF=2\OKbasis$ if $v\ndivides 2$ and $\UF=\OKbasis-\ol{\OKbasis}$ if $v|2$.
Let $t=t_v=\OKbasis+\ol{\OKbasis}$ and $\addchar^\circ(x):=\addchar(-d_F^{-1}x)$. For brevity, we drop the subscript and simply write $\adsx$ for $\adsx_v$.
Since $\adsx$ is self-dual, $\adsx|_{\cO_F^\x}=1$ and $\nads(\uf)=-1$. By \cite[Prop.\,4.5 (1-3)]{Hsieh:Hecke_CM}, the formula of $A_\beta(\adsx)$ is given as follows.
\begin{mylist}
\item
If $v(\twobeta)\geq 0$ and $v(\twobeta)\not\con 0\pmod{2}$, then
\[A_\beta(\adsx)=\addchar^\circ(2^{-1}t\beta)\cdot (-1)^{v(2)+1}(1+\abs{\uf}).\]
\item If either $v(\twobeta)<-1$ or $v(\twobeta)\con 0\pmod{2}$, then $A_\beta(\adsx)=0$.
\item If $v(\twobeta)=-1$, then
\[A_\beta(\adsx)=\addchar^\circ(2^{-1}t\beta)\cdot \abs{\uf}\sum_{x\in \bfk_F}\adsx^{-1}(x+\OKbasis)\addchar^\circ(\beta x).\]
\end{mylist}
It follows immediately that $v_p(A_\beta(\adsx))\geq v_p(1+\abs{\uf}^{-1})$ if $v(\twobeta)\not =-1$, and $v_p(A_\beta(\adsx))\geq \mu_p(\adsx)$ if $v(\twobeta)=-1$. On the other hand, note that $p$ divides $1+\abs{\uf}^{-1}$ if $\mu_p(\adsx)>0$. Thus $v_p(A_\beta(\adsx))\geq \mu_p(\adsx)$ for every $\beta\in F^\x$.

We proceed to prove the second assertion. Choose a sufficiently large finite extension $L$ of $\Qp$ so that $\adsx$ and $A_\beta(\adsx)$ for $\beta\in \uf^{-1}\cO^\x_F$ take value in $L$.
Let $e_L=v_L(p)$ and let \[m=\inf_{x\in\cO_F}v_L(\adsx(x+\OKbasis)-1)=e_L^{-1}\mu_p(\adsx).\] We define the function $f:\bfk_F\to\bfk_L\subset\Fpbar$ by
\[f(\ol{x})=\uf_L^{-m}(\adsx(x+\OKbasis)-1)\pmod{(\uf_L)}.\] For $\gamma\in \cO_F$, define $\addchar_{\ol{\gamma}}:\bfk_F\to \Fpbar$ by $\addchar_{\ol{\gamma}}(x)=\addchar^\circ(\frac{\gamma}{\uf}x)\pmod{\frakm}$.
Then $\stt{\addchar_{\ol{\gamma}}}_{\ol{\gamma}\in \bfk_F}$ gives a $\Fpbar$-basis of the space of $\Fpbar$-valued functions on $\bfk_F$. Then $f$ can be uniquely written as
$f(x)=\sum_{\ol{\gamma}\in\bfk_F}c_\gamma(f)\addchar_{\ol{\gamma}}(x)$, where $c_\gamma(f)$ is the $\ol{\gamma}$-coefficient of $f$ given by
\begin{align*}c_\gamma(f)&=\abs{\uf}\sum_{x\in\bfk_F}f(x)\psi_{\ol{\gamma}}(x)=\uf_L^{-m}\abs{\uf}\sum_{x\in\bfk_F}\adsx^{-1}(x+\OKbasis)\psi^\circ(\frac{\gamma}{\uf}x)\\
&=\uf_L^{-m}\addchar^\circ(-2^{-1}t\beta)\cdot A_{\gamma\uf^{-1}}(\adsx)\pmod{\frakm_L}.\end{align*}
Since $f$ is a non-zero function by definition, some $\ol{\gamma}$-coefficient of $f$ is nonzero, namely  $c_{\gamma}(f)\not\con 0\pmod{\frakm_L}$.
Let $\Beth:=\gamma \uf^{-1}$. Then $v_L(A_{\Beth}(\adsx))=m$ and hence $v_p(A_{\Beth}(\adsx))=\mu_p(\adsx)$.
\end{proof}

The following proposition is the key ingredient in our proof.
\begin{prop}\label{P:NVAbeta.V}
There exists $\Beth_v\in F^\x$ such that
\begin{itemize}
\item[(i)] $v_p(A_{\Beth_v}(\adsx_v))=\mu_p(\adsx_v)$,
\item[(ii)] $W(\nads_v)\tau_{E/F}(\Beth_v)=\nads_v(2\skewhf)$.
\end{itemize}
\end{prop}
\begin{proof}
When $w(\frakC^-)=1$ and $v$ is inert, (i) is verified in \propref{P:6.Theta}. Suppose that either $w(\frakC^-)>1$ or $v$ is ramified. Then we must have $\mu_p(\adsx_v)=0$ as $v\ndivides p$ and $p>2$. By \cite[Lemma 6.4]{Hsieh:Hecke_CM}, there exists $\Beth_v\in F^\x$ such that $A_{\Beth_v}(\adsx_v)\not\con 0\pmod{\frakm}$. Thus $v_p(A_{\Beth_v}(\adsx_v))=\mu_p(\adsx_v)=0$. To show the epsilon dichotomy property (ii) for this $\Beth_v$, we
note that (i) implies that $A_{\Beth_v}(\adsx_v)\not =0$ ($\adsx_v$ is ramified), and (ii) follows from \lmref{L:key.V}.
\end{proof}
\begin{Remark}In virtue of \cite[Prop.\,6.7]{Hsieh:Hecke_CM}, \propref{P:NVAbeta.V} removes the assumption (C) in \cite[Thm.\,6.8]{Hsieh:Hecke_CM}.
\end{Remark}
\subsection{}
Now we are ready to prove our main theorem.
\begin{thm}\label{T:main.V}Suppose that $p\ndivide D_{\cF}$. Let $\adsx$ be a self-dual Hecke character of $\cK^\x$ such that \HypMu
Then
\[\Imu^-=\sum_{v|\frakC^-}\mu_p(\adsx_v).\]
\end{thm}
\begin{proof}
In view of \eqref{E:FA.V} and \propref{P:6.Theta}, we find that \[v_p(\bfa_\beta^\setp(\adsx,\frakc(a)))\geq \sum\limits_{v|\frakC^-}\mu_p(\adsx_v)\text{ for all }\beta\in\cF_+\text{ and }a\in\cD_1.\]
Combined with the formula \eqref{E:fourierformula.E} of $\bfa_\beta(\Eadsu,\frakc(a))$ and \thmref{T:1.V}, this implies that
\beq\label{E:10.V}\Imu^-\geq \sum_{v|\frakC^-}\mu_p(\adsx_v).\eeq

For each $v|\frakC^-$, we let $\Beth_v$ be as in \propref{P:NVAbeta.V}. Then
$v_p(A_{\Beth_v}(\adsx_v))=\mu_p(\adsx_v)$ and $W(\nads_v)\qchKF(\Beth_v)=\nads_v(2\skewhf)$ for every $v|\frakC^-$.
From the assumption that $W(\nads)=\prod_v W(\nads_v)=1$ we can deduce that there exists $\beta\in\cF_+$ such that
\begin{enumerate}
\item $\beta\in\cO_{\cF,(p\Csplit\Csplit^c)}^\x$,
\item $A_\beta(\adsx_v)=A_{\Beth_v}(\adsx_v)$ for every $v|\frakC^-$,
\item $\prod_{\frakq|\frakC^-}\frakq^{v_\frakq(\beta)}=(\beta)\frakc(\OK)\rmN_{\cK/\cF}(\fraka)$ for some prime-to-$p\frakC$ ideal $\fraka$ of $\OK$.
\end{enumerate}
(\cf \cite[Prop.\,6.7]{Hsieh:Hecke_CM}.) Let $\bfc\in\AFf^\x$ be the idele such that $\bfc_v=\beta^{-1}$ for all $v\ndivide p\frakC\frakC^c$ and $\bfc_v=1$ if $v|p\frakC\frakC^c$. Then $\frakc(\fraka):=\il_\cF(\bfc)=\frakc(\OK)\rmN_{\cK/\cF}(\fraka)$ is the ideal corresponding to $\bfc$. Let $u\in\cU_p$ such that $u\con\beta\pmod{p}$. By \eqref{E:FA.V} and \eqref{E:fourierformula.E}, we find that
\beq\label{E:11.V}v_p(\bfa_\beta(\Eadsu,\frakc(\fraka)))=\sum_{v|\frakC^-}v_p(A_\beta(\adsx_v))=\sum_{v|\frakC^-}v_p(A_{\Beth_v}(\adsx_v))=\sum_{v|\frakC^-}\mu_p(\adsx_v).\eeq
Combining \thmref{T:1.V}, \eqref{E:10.V} with \eqref{E:11.V}, we obtain
\[\Imu^-=\sum_{v|\frakC^-}\mu_p(\adsx_v).\qedhere\]
\end{proof}

\bibliographystyle{amsalpha}
\bibliography{C:/users/MingLun/texsetting/mybib}
\end{document}